\documentclass[a4paper,12pt]{article}
\usepackage[top=2.5cm,bottom=2.5cm,left=2.5cm,right=2.5cm]{geometry}
\usepackage{cite, amsmath, amssymb}
\usepackage{amssymb}
\usepackage{graphicx}
\newenvironment{proof}{{\noindent \it Proof.}}{\hfill $\blacksquare$\par}
\usepackage{latexsym}
\usepackage{caption}
\usepackage{textcomp}
\usepackage{graphicx,booktabs,multirow}
\usepackage{tikz}
\usetikzlibrary{calc,arrows,positioning,shapes}
\usepackage{bigdelim}
\usepackage{nicematrix}
\usepackage{mathtools}
\usepackage{comment}
\usepackage{makecell}
\usetikzlibrary{decorations.pathreplacing}
\usetikzlibrary{intersections}
\usepackage{enumerate}
\usepackage{graphicx,booktabs,multirow}
\usepackage{appendix}
\newtheorem{theorem}{Theorem}[section]
\newtheorem{proposition}[theorem]{\rm\bfseries Proposition}
\newtheorem{Observation}[theorem]{\rm\bfseries Observation}
\newtheorem{lemma}[theorem]{Lemma}

\newtheorem{remark}[theorem]{Remark}
\newtheorem{conjecture}[theorem]{Conjecture}

\newtheorem{Proof of Theorem 1.7.}[theorem]{Proof of Theorem 1.7.}

\newtheorem{definition}[theorem]{Definition}
\usepackage[section]{placeins}
\def\NAT@def@citea{\def\@citea{\NAT@separator}}
\allowdisplaybreaks

\begin{document}
	\vspace*{10mm}
	
	\noindent
	{\Large \bf  The saturation number of $K^s_t$}
	
	\vspace*{7mm}
	
	\noindent
	{\large \bf Xinghui Zhao,  Lihua You*,  Xiaoxue Zhang}
	\noindent
	
	\vspace{7mm}
	
	\noindent
	School of Mathematical Sciences, South China Normal University,  Guangzhou, 510631, P. R. China.

	\noindent
	e-mail: {\tt 2022021990@m.scnu.edu.cn (X. Zhao)},\quad{\tt ylhua@scnu.edu.cn (L. You)},
	
	\quad{\tt zhang\_xx1209@163.com (X. Zhang)}.
	
	\noindent
	$^*$ Corresponding author
	\noindent
	
	\vspace{7mm}

	\noindent
	{\bf Abstract} \ 
	\noindent
	For a given  graph $F$, a graph $G$ is said to be $F$-saturated if $G$ contains no copy of $F$ but for any edge $uv\notin E(G)$, $G+uv$ contains a copy of $F$. The saturation number $sat(n,F)$ is defined as the minimum number of edges among all $n$-vertex $F$-saturated graphs. The virus graph $K^s_t$, where $s\geq0$ and $t\geq \max\{3,s\}$, is a graph of order $s+t$ constructed by attaching $s$ distinct leaves to $s$ different vertices of a complete graph $K_t$. Hua and Peng [Discrete Math. 349 (2026) 114674] determined $sat(n,K^2_3)$ and characterized its corresponding extremal graphs. In this paper, we determine $sat(n,K^3_3)$ and  $sat(n,K^2_t)$ with $t\geq 4$,  together with the structural descriptions of the related extremal saturated graphs.
	\\[2mm]
	

	\baselineskip=0.30in
	
	\section{Introduction}
	
	\hspace{1.5em}Throughout this paper, we consider  simple  graphs. Let $G=(V(G),E(G))$ be a graph with vertex set $V(G)$ and edge set $E(G)$ with $|E(G)|=e(G)$.  We write  $uv \in E(G)$ if the vertices $u$ and $v$ are adjacent in $G$, and $uv \not\in E(G)$ otherwise. Let $v\in V(G)$, $U\subset V(G)\setminus \{v\}$ be a non-empty vertex set. If there exists $u\in U$ such that $uv\in E(G)$,  we denote $v\sim U$, and $v\not\sim U$ otherwise. Let $N_G(u)$ ($N(u)$ for short) be the neighborhood set of $u$ in $G$, $N_G[u]=N_G(u)\cup \{u\}$ ($N[u]$ for short). Then the degree of the vertex $v$ in $G$ is equal to $|N_G(v)|$ (the cardinality of the set $N_G(v)$), denoted by $d_G(v)$ ($d(v)$ for short). We use $\delta(G)$  to denote the minimum degree  of $G$. For $u,v\in V(G)$, the \emph{distance} between $u$ and $v$ is the length of the shortest path connecting $u$ and $v$ in $G$, denoted by $d_G(u,v)$ ($d(u,v)$ for short).  The \emph{eccentricity} $\varepsilon_{G}(v)$ (for short $\varepsilon(v)$) of the vertex $v\in V(G)$ is given by
	$\varepsilon_{G}(v) = \max\{d(u, v) | u\in V(G)\}$. The \emph{diameter} of $G$ is the largest eccentricity among all vertices in $G$, denoted by $diam(G)$. Clearly, $\varepsilon(v)\leq diam(G)$ for any $v\in V(G)$.
	
	Let $G$ be a graph with  $uv\notin E(G)$,   we use $G+uv $ to denote a new graph from $G$ by adding    edge $uv$  to  $G$, and $kG$ the disjoint union of $k$ copies of graph $G$. Given two vertex-disjoint graphs $G_1$ and $G_2$, we denote by $G_1\cup G_2$ the disjoint union of the two graphs,   and $G_1\vee G_2$ the graph obtained from $G_1\cup G_2$ by joining each vertex of $G_1$ with each of $G_2$. 
	
	Usually,  $K_n$  denote a complete graph  of order $n$. The \emph{virus}, denoted by $K^s_t$ with  $s\geq0$ and  $t\geq \max\{3,s\}$ , is the graph of order $s+t$ obtained by attaching $s$ leaves to $s$ distinct vertices of $K_t$.

	
	Given a family of graphs $\mathcal{F}$, a graph is $\mathcal{F}$-free if it contains no  copy of member of $\mathcal{F}$. A graph $G$ is called  $\mathcal{F}$-\emph{saturated} if $G$ is $\mathcal{F}$-free but $G+e$ is not for any $e\in E(\overline{G})$.  The \emph{saturation number} of $\mathcal{F}$, denoted by $\mathrm{sat}(n,\mathcal{F})$, is the minimum number of edges
	among all $n$-vertex $\mathcal{F}$-saturated graphs. Further, $\mathrm{SAT}(n,\mathcal{F})=\{G\mid |V(G)|=n, e(G)=\mathrm{sat}(n,\mathcal{F}), \text{and $G$ is $\mathcal{F}$-saturated}\}$. Similarly,  the \emph{connected saturation number}  $\mathrm{csat}(n,\mathcal{F})$ is the minimum number of edges among all $n$-vertex connected  $\mathcal{F}$-saturated graphs, and $\mathrm{CSAT}(n,\mathcal{F})=\{G\mid |V(G)|=n, e(G)=\mathrm{csat}(n,\mathcal{F}), \text{and $G$ is  connected  $\mathcal{F}$-saturated graph}\}$.  If $\mathcal{F}=\{F\}$, then we write $\mathrm{sat}(n,\mathcal{F})$ as $\mathrm{sat}(n,F)$, $\mathrm{csat}(n,\mathcal{F})$ as $\mathrm{csat}(n,F)$,   $\mathrm{SAT}(n,\mathcal{F})$ as $\mathrm{SAT}(n,F)$, and $\mathrm{CSAT}(n,\mathcal{F})$ as $\mathrm{CSAT}(n,F)$, respectively. In 1964, Erd\H{o}s,  Hajnal, and  Moon \cite{e2} presented the first result on saturated numbers as follows.
		\begin{theorem}{\rm(\!\!\cite{e2})}\label{e2}
		If $3\leq \alpha\leq n$, then $\mathrm{sat}(n,K_{\alpha})=e(K_{\alpha-2}\vee (n-\alpha+2)K_1)$. The only $n$-vertex $K_{\alpha}$-saturated graph with $\mathrm{sat}(n,K_{\alpha})$ edges is the graph $K_{\alpha-2}\vee (n-\alpha+2)K_1$.
	\end{theorem}
	
	 Since then, extensive research has been conducted on the saturation number for various graph families $\mathcal{F}$, such as cycles \cite{Chen2011,Lan,Ollmann,Tuza}, complete bipartite graphs \cite{Chen2014,Huang,Ollmann}, unions of cliques \cite{Chen,Fau,Li,Zhu}, wheel graphs \cite{Qiu,Song}, books \cite{cgt}, and join graphs \cite{Hu2025,Hu,z}, among others. 
	
	Clearly, a connected graph of order at least $t+1$ is $K_t^1$-saturated if and only if it is $K_t$-saturated. 
	For this reason, we focus our attention on $K_t^s$-saturated graphs with $s\ge 2$. In 2026, Hua and Peng \cite{Hua} determined the saturation number of $K^2_3$ and characterized the  corresponding extremal graphs. Motivated  by the above results, in this paper, we study the saturation numbers of $K^3_3$ and $K^2_t$ with $t\geq4$. 
	
	Before presenting the main results of this paper, we first introduce a special family of graphs. Let $H$ be a connected graph,  $\mathcal{M}_n^t(H)$ be the family of graphs of order $n$ constructed from $H$ such that each graph $G\in\mathcal{M}_n^t(H)$ is obtained by attaching cliques to vertices of $H$, so that   for any vertex $v_i\in V(H)$, there exist nonnegative integers $p_i$ and $q_i$ such that $v_i$ is adjacent to all vertices of each of the $p_i$ copies of $K_{t-1}$, and to exactly $t-2$ vertices of each of the $q_i$ copies of $K_t$, and $p_i,q_i$ satisfy the following requirements: (1) if $H\cong K_2$ and $q_i=0$, then $p_i\neq 1$; (2) if $H\cong K_{t-2}\vee 2K_1$ with $t\geq4$, then there exist at least two distinct vertices $v_i,v_j\in V(H)$ such that $p_i+q_i\neq0$ and $p_j+q_j\neq0$; (3) if $H\cong K_2\vee 3K_1$, then either there exist two distinct vertices $v_i,v_j\in V(H)$ such that $p_i+q_i\neq0$ and $p_j+q_j\neq0$,  or there exists $v_k\in V(H)$ with $d_H(v_k)=4$ such that $p_k+q_k\neq0$.  In particular, if $q_i=0$ for every $v_i\in V(H)$, we write $\mathcal{M}_n^t(H)$ as $\mathcal{SM}^t_n(H)$; for $t=4$, we write $\mathcal{M}_n^4(H)$ as $\mathcal{M}_n(H)$ for short.   For clarity, we present an illustration of such a graph in Figure \ref{f1}.

	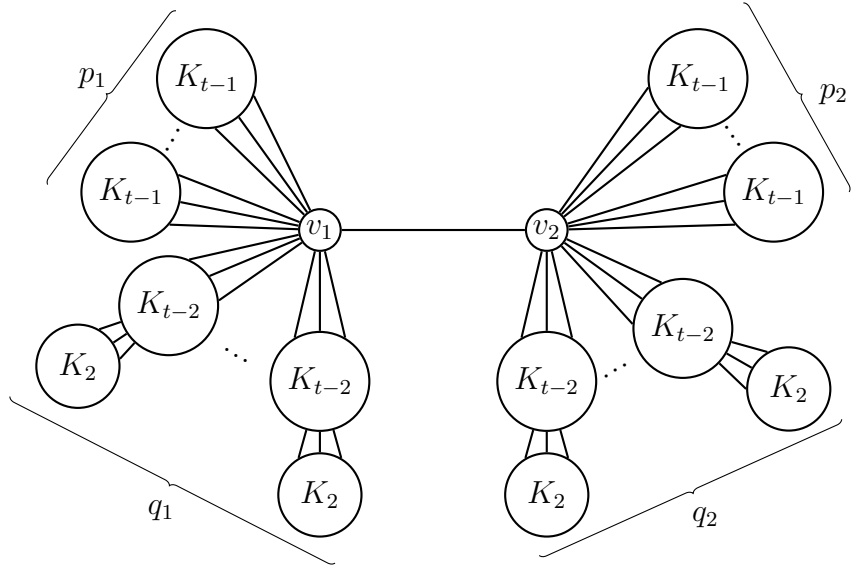
\begin{figure}[ht]
		\centering
		\begin{tikzpicture}[
			scale=1,
			round/.style={circle, draw, thick, minimum size=11mm},  smallv/.style={
				circle, 
				draw, 
				thick, 
				inner sep=0pt,     
				minimum size=5.5mm 
			}
			]
			
			\node[smallv] (v1) at (0, 0)  {$v_1$};
			\node[smallv] (v2) at (3, 0)  {$v_2$};
			\draw[thick] (v1)--(v2);
			
			\node[round] (A) at (-1.5, 2) {$K_{t-1}$};
			\node[round] (B) at (-2.5, 0.5)   {$K_{t-1}$};
			
			\draw[thick] (v1)--(A.280);
			\draw[thick] (v1)--(A.340);
			\draw[thick] (v1)--(A.310);
			
			\draw[thick] (v1)--(B);
			\draw[thick] (v1)--(B.20);
			\draw[thick] (v1)--(B.320);
			
			\node[rotate=150] at (-2, 1.1) {$\vdots$}; 
			\node at (-3, 2) {$p_1$};     
			
			\node[round] (C) at (-2, -1) {$K_{t-2}$};  
			\node[round] (D) at (-3.2, -1.8) {$K_2$};  
			
			\node[round] (I) at (0, -2) {$K_{t-2}$};  
			\node[round] (J) at (0, -3.5) {$K_2$};   
			
			\draw[thick] (C)--(D);
			\draw[thick] (C)--(D.60);
			\draw[thick] (C)--(D.10);
			\draw[thick] (I)--(J.60);
			\draw[thick] (I)--(J.90);
			\draw[thick] (I)--(J.120);
			
			\draw[thick] (v1)--(C.6);
			\draw[thick] (v1)--(C.36);
			\draw[thick] (v1)--(C.66);
			\draw[thick] (v1)--(I.60);
			\draw[thick] (v1)--(I.90);
			\draw[thick] (v1)--(I.120);
			
			\node[rotate=60] at (-1.2, -1.6) {$\vdots$}; 
			\node at (-2.1, -3.7) {$q_1$};     
			
			\node[round] (E) at (5, 2) {$K_{t-1}$};
			\node[round] (F) at (6, 0.5)   {$K_{t-1}$};
			
			\draw[thick] (v2)--(E.220);
			\draw[thick] (v2)--(E.250);
			\draw[thick] (v2)--(E.190);
			
			\draw[thick] (v2)--(F.190);
			\draw[thick] (v2)--(F.160);
			\draw[thick] (v2)--(F.220);
			
			\node[rotate=35] at (5.4, 1.33) {$\vdots$};
			\node at (6.8, 1.8) {$p_2$};     
			
			\node[round] (G) at (4.8, -1.3) {$K_{t-2}$};
			\node[round] (H) at (6.2, -2.1) {$K_2$};    
			\node[round] (K) at (3, -2) {$K_{t-2}$};
			\node[round] (L) at (3, -3.5) {$K_2$};

			\draw[thick] (G)--(H);
			\draw[thick] (G)--(H.120);
			\draw[thick] (G)--(H.180);
			\draw[thick] (K)--(L.60);
			\draw[thick] (K)--(L.90);
			\draw[thick] (K)--(L.120);
			
			\draw[thick] (v2)--(G);
			\draw[thick] (v2)--(G.115);
			\draw[thick] (v2)--(G.175);
			\draw[thick] (v2)--(K.60);
			\draw[thick] (v2)--(K.90);
			\draw[thick] (v2)--(K.120);
			
			\node[rotate=300] at (4, -1.8) {$\vdots$};
			\node at (5.1, -3.8) {$q_2$};     
			\draw[decorate,decoration={brace}] (-3.6,0.6) -- (-1.9,2.9);
			\draw[decorate,decoration={brace}] (0.2,-4.4) -- (-4.1,-2.2);
			\draw[decorate,decoration={brace}] (6.9,-2.5) -- (2.9,-4.3);
			\draw[decorate,decoration={brace}] (5.6,3) -- (7,0.5);
		\end{tikzpicture}
		\caption{A graph in $\mathcal{M}_n^t(K_2)$.}\label{f1}
	\end{figure}

	Our main results are as follows.
	
		\begin{theorem}\label{t2}
		Let $n\geq 6$, $\mathcal{G}=   (\mathcal{M}_q(K_2)\cup \frac{n-q}{5}K_5)\cup (\mathcal{M}_q(K_3)\cup\frac{n-q}{5}K_5)\cup( \mathcal{M}_p(K_2\vee (k-2)K_1)\cup\frac{n-p}{5}K_5)$ with $q\leq n$, $6 \leq p\leq n$, $4\leq k\leq p$ and  $\mathcal{M}_q(H)\cup H^*=\{\overline{H}\cup H^*\mid \overline{H}\in \mathcal{M}_q(H)\}$. Then  $\mathrm{sat}(n,K^2_4)=2n-3$ and $\mathrm{SAT}(n,K^2_4)=\mathcal{G}$.
	\end{theorem}
	
	\begin{theorem}\label{t3}
		Let $n,t,r$ be positive integers with  $n\geq t+2\geq7$ and $2\leq r\leq t-1$.  Then

		\begin{equation}\label{s11}
			\mathrm{csat}(n,K^2_t)=\begin{cases}
				\frac{tn}{2}-\frac{t+2}{2},&\text{if}\  n\equiv 1\ (\mathrm{mod}\ t-1) \ \text{and} \ n\neq2t-1; \\
				\frac{tn}{2}-\frac{r(t-r+1)}{2},&\text{if}\  n\equiv r\ (\mathrm{mod}\ t-1);\\
				t^2-t,&\text{if}\ n=2t-1.
			\end{cases}	\end{equation}  and 
		\begin{equation}\label{s12}
			\mathrm{CSAT}(n,K^2_t)=\begin{cases}
				\mathcal{SM}^t_{n}( K_{t-2}\vee 2K_1),\ &\text{if}\  n\equiv 1\ (\mathrm{mod}\ t-1) \ \text{and} \ n\neq2t-1; \\
				\mathcal{SM}^t_{n}(K_r), \ &\text{if}\  n\equiv r\ (\mathrm{mod}\ t-1);\\
				\mathcal{SM}^t_{n}(K_1), \ &\text{if}\ n=2t-1.
			\end{cases}
		\end{equation}
	
	\end{theorem}
	
	 \begin{theorem}\label{t4}
		Let $t,n,k_1,f$ be positive integers with $t\geq5$, $1\leq f\leq t-1$ and $n=f+k_1(t-1)\geq \frac{(t+3)(t-1)+t+1}{2}$, and $$g(n,t)=\begin{cases}
			\frac{tn}{2}-\frac{t(t+2)}{8},&\text{if}\ t \ \text{is even}; \\
			\frac{tn}{2}-\frac{(t+1)^2}{8},&\text{if}\  t\ \text{is odd and }  f+\frac{t+1}{2}\equiv 0\ (\mathrm{mod}\ 2);\\
			\frac{tn}{2}-\frac{(t+3)(t-1)}{8},&\text{if}\  t\ \text{is odd and }  f+\frac{t+1}{2}\equiv 1\ (\mathrm{mod}\ 2).\nonumber
		\end{cases}$$ Then 
		$\mathrm{sat}(n,K^2_t)=g(n,t)$.
	\end{theorem}

	\begin{theorem}\label{t1}
	Let    $n\geq 7$. Then $\mathrm{sat}(n,K^3_3)=\begin{cases}
		\frac{3n-3}{2},&\text{if}\ n \ \text{is odd}; \\
		\frac{3n}{2},&\text{if}\ n\ \text{is even}.
	\end{cases}$ Moreover,  if $n$ is odd, then $\mathrm{SAT}(n,K^3_3)=\{K_1\vee\frac{n-1}{2}K_2\}$; if $n$ is even, then $K_1\vee (\frac{n-4}{2}K_2\cup K_3)\in \mathrm{SAT}(n,K^3_3)$.
	\end{theorem}
	
	The organization of this paper is as follows.  In Section 2, we investigate the properties of $K^2_t$-saturated graphs and present proofs of Theorems \ref{t2}--\ref{t4}.  In Section 3, we obtain some properties of $K^3_3$-saturated graphs and present a proof of  Theorem \ref{t1}. In Section 4, we summarize and discuss the work in this paper, and propose two conjectures for further research.

		\section{The saturation number of $K^2_t$ with $t\geq4$}
	
\hspace{1.5em} In the rest of the paper, we stipulate that  $G$ is not $F$-saturated if $|V(G)| < |V(F)|$, and we also write  $\{1,\ldots,t\}$ as $[t]$, 	 refer to $K_t$  as the center part of $K^s_t$, and use $V(CK^s_t)$ to denote the set of  vertices in the center part of $K^s_t$.   

 In this section, we prove Theorems \ref{t2} and \ref{t3}. 
	
	\begin{proposition}\label{o2}
		Let  $t\geq4$, $G$ be a connected $K^2_t$-saturated graph of order $n$, $V_1=\{v_1,v_2,\ldots,v_t\}\subset V(G)$, and $G[V_1]\cong K_t$. Then one of the following cases  holds:
		
		{\rm \item(i)} there exists $i\in [t]$ such that $v_i$ is a cut vertex and $d(v_j)=t-1$ for all $j\in [t]\setminus\{i\}$;
		
		{\rm \item(ii)} there exist distinct $i,j\in [t]$ such that  $d(v_i)=d(v_j)=t-1$,  $d(v_k)=t$ for all  $k\in [t]\setminus\{i,j\}$ with $vv_k\in E(G)$ for some $v\in  V(G)\setminus V_1$.
	\end{proposition}
	
	\begin{proof}
	Clearly,  there exists $v\in V(G)\setminus V_1$ such that $v\sim V_1$ since $G$ is connected and $n\geq t+2$.
		
		If there exists $u\in V(G)\setminus (V_1\cup \{v\})$ such that $u\sim V_1$, then $vv_i,uv_i\in E(G)$ for some $i\in [t]$ and $d(v_j)=t-1$ for all $j\in[t]\setminus\{i\}$ since $G$ is $K^2_t$-free, which implies that $v_i$ is a cut vertex of $G$, i.e., (i) holds.
		
		If exactly one vertex $v$ is adjacent to $V_1$, 
		we may assume that $vv_1,\ldots,vv_s\in E(G)$ and $vv_{s+1},\ldots,vv_t\notin E(G)$, where $1\leq s\leq t$. If $1\leq s\leq t-3$, then $G+vv_{s+1}$ contains no copy of $K^2_t$, a contradiction. If $s\geq t-1$, then $G$ contains  a copy of $K^2_t$ since $G$ is connected and $n\geq t+2$, a contradiction. Thus $s=t-2$, and (ii) holds.
		
		We complete the proof.
	\end{proof}
	
	\begin{definition}\label{d2}
		Let   $t\geq4$, $G$ be a connected $K^2_t$-saturated graph of order $n$, $V_1=\{v_1,v_2,\ldots,v_t\}\subset V(G)$, and $G[V_1]\cong K_t$.	If $v_1,v_2,\ldots,v_t$ satisfy (i) Proposition \ref{o2}, we call $G[V_1]$ a type-I  $K_t$, $G[V_1\setminus\{v_i\}]$ an unlucky  $K_{t-1}$, and $v_i$ a lucky vertex; if $v_1,v_2,\ldots,v_t$ satisfy (ii) of Proposition \ref{o2}, we call $G[V_1]$ a type-II   $K_t$.
	\end{definition}
	
	By Proposition \ref{o2} and Definition \ref{d2}, we know that the vertex sets of a type-I $K_t$ and a type-II $K_t$ are disjoint.
	
	\begin{lemma}\label{l4}
		Let  $ t\geq4$, $G$ be a connected $K^2_t$-saturated graph of order $n$, and $G'$  be the graph obtained from $G$ by deleting all vertices of every type-II $K_t$ and every unlucky $K_{t-1}$. Then $G'$ is a  connected $K_t$-free graph with $	e(G)\geq e(G')+\frac{t}{2}|V(G)\setminus V(G')|$, and equality holds if and only if $t=4$ or every $K_t$ in $G$ is a type-I   $K_t$.
	\end{lemma}
	
	\begin{proof}
		Clearly, 	$G'$ is  connected by the definition of $G'$.  If $K_t$ is a subgraph of $G'$, then there exists a set $U\subseteq V(G')$ such that $G'[U]\cong K_t$. Thus $G[U]$  is also a subgraph of $G$, which implies that $G[U]$ is a type-I  $K_t$ or a type-II  $K_t$, a contradiction to the definition of $G'$. Therefore, $G'$ is a  connected $K_t$-free graph. 
		
		Without loss of generality, suppose that $G$ contains $b$ unlucky  $K_{t-1}$ and $c$ type-II   $K_t$ with $b,c\geq0$. Then  $e(G)=e(G')+b\frac{t(t-1)}{2}+c(\frac{t(t-1)}{2}+(t-2))=e(G')+\frac{t}{2}(b(t-1)+c(t+\frac{t-4}{t})) \geq e(G')+\frac{t}{2} |V(G)\setminus V(G')|$ since $t\geq4$ and $|V(G)\setminus V(G')|=b(t-1)+ct$, and   equality holds if and only if $t=4$ or $c=0$.
	\end{proof}

	\begin{lemma}\label{l12}
		Let  $t\geq4$, $G$ be a connected $K^2_t$-saturated graph of order $n$. Then there exists $v\in V(G)$ such that  $d_{K^1_t} (v)\neq t-1$ for any  subgraph $K^1_t$ of $G$.
	\end{lemma}
	
	\begin{proof}
		Clearly, if $G$ contains no copy of $K_t$, then the result holds.
		
		If $G$ contains a copy of $K_t$, say $Q$,  then  $Q$ is either a type-I  $K_t$ or a type-II  $K_t$ by Proposition \ref{o2} and Definition \ref{d2}. If  $Q$ is  a type-I  $K_t$, then there exists $v\in V(Q)$ such that $v$ is a lucky vertex of $Q$, and thus $d_{K^1_t} (v)\neq t-1$ for any  subgraph $K^1_t$ of $G$. If $Q$ is  a type-II  $K_t$, then there exists unique vertex $v\in V(G)\setminus V(Q)$ such that $v\sim V(Q)$, and thus $d_{K^1_t} (v)\neq t-1$ for any  subgraph $K^1_t$ of $G$.
		
	 We complete the proof.
	\end{proof}
	
	\begin{lemma}\label{l13}
		Let  $t\geq4$, $G\cong \bigcup\limits_{i=1}^{l} G_i$ be a disconnected $K^2_t$-saturated graph of order $n$ with $l\geq2$. Then $G\cong G_1\cup (l-1) K_{t+1}$, where $G_1$  is either a $K^2_t$-saturated graph,    or isomorphic to 
		$K_{f}$ with $1\leq f\leq t+1$.
	\end{lemma}
	
	\begin{proof}
Clearly, every $G_i$ is either a $K^2_t$-saturated graph, or isomorphic to $K_f$ with $1\leq f\leq t+1$.

  If there exist two distinct connected components $G_1$ and $G_2$ such that both $G_1$ and $G_2$ are either $K^2_t$-saturated graphs or isomorphic to $K_{f}$ with $1\leq f\leq t$, then there exist $v_1\in V(G_1)$ and $v_2\in V(G_2)$ such that $d_{K^1_t}(v_1)\neq t-1$ for any  subgraph $K^1_t$ of $G_1$ and $d_{K^1_t}(v_2)\neq t-1$ for any  subgraph $K^1_t$ of $G_2$  by Lemma \ref{l12} and $f\leq t$, and thus $G+v_1v_2$ contains no copy of $K^2_t$, a contradiction.
Therefore, $G\cong G_1\cup (l-1) K_{t+1}$, where $G_1$  is either a $K^2_t$-saturated graph,    or isomorphic to 
$K_{f}$ with $1\leq f\leq t+1$.
		\end{proof}




	Now, we show Theorems \ref{t2} and \ref{t3}, respectively.
	
	\noindent\textbf{\textit{Proof of Theorem \ref{t2}.}} If	$G\in \mathcal{G}$ with  $q\leq n$, $6 \leq p\leq n$ and  $4\leq k\leq p$, then $G$ is $K^2_4$-saturated and $e(G)=2n-3$ by direct inspection and calculation. Thus we have $\mathrm{sat}(n,K^2_4)\leq 2n-3$. 
	
Now we show $\mathrm{sat}(n,K^2_4)\geq 2n-3$ and $\mathrm{SAT}(n,K^2_4)=\mathcal{G}$. 	Let $G$ be a $K^2_4$-saturated graph of order $n\geq6$. 
It suffices to verify $e(G)\geq 2n-3$, and $G\in\mathcal{G}$ if equality holds.
	
	
		\textbf{Case 1.} $G$ is connected. 
		
			\textbf{Subcase 1.1.} $G$ contains no  copy of $K_4$.
			
			Clearly, $G$ is a $K_4$-saturated graph	since $G$ is a connected $K^2_4$-saturated graph. Then $e(G)\geq 2n-3$ with equality hold if and only if $G\cong K_2\vee (n-2)K_1$ by Theorem \ref{e2}.
			
			\textbf{Subcase 1.2.} $G$ contains  a copy of $K_4$.
			
				By Proposition \ref{o2} and Definition \ref{d2}, each copy of $K_4$ in $G$ is either a type-II $K_4$	or contains an unlucky $K_3$. Let $G'$ be defined as Lemma \ref{l4}. Then $G'$ is not a complete graph or $G'\cong K_f$ with $1\leq f\leq 3$ by Lemma \ref{l4}.
				
					If $G'\cong K_f$, then $e(G)=\frac{f(f-1)}{2}+2(n-f)=2n-\frac{f(5-f)}{2}\geq 2n-3$ by Lemma \ref{l4}, and  equality holds if and only if $f\in\{2,3\}$. Combining  the fact that $G$ is  $K^2_4$-saturated, we have  $G\in \mathcal{M}_n(K_2)\cup \mathcal{M}_n(K_3)$ if equality holds.

					If $G'$ is not a complete graph, then $G'$ is $K_4$-saturated with $|V(G')|\geq4$. Otherwise, there exists  an edge $e\in E(\overline{G'})\subseteq E(\overline{G})$ such that $G'+e$  contains no  copy of $K_4$, so $G+e$ contains no  copy of $K^2_4$,  which yields a contradiction. Thus $e(G)\geq 2|V(G')|-3+2|V(G)\setminus V(G')|=2n-3$  by  Lemma \ref{l4} and Theorem \ref{e2}, and equality holds if and only if $G'\cong  K_2\vee (k-2)K_1$ with $4\leq k\leq n$. If $G'\cong  K_2\vee 2K_1$ and there exists exactly one vertex $v_i\in V(G')$ such that $p_i+q_i\neq0$, then $G$ is not $K^2_4$-saturated, a contradiction. If $G'\cong K_2\vee 3K_1$, and neither there exist two distinct vertices $v_1,v_2\in V(G')$ such that $p_1+q_1\neq 0$ and $p_2+q_2\neq 0$, nor there exists a vertex $v_3\in V(G')$ with $d_{G'}(v_3)=4$ such that $p_3+q_3\neq 0$, then $G$ is not $K_4^2$-saturated, a contradiction. Thus we have $G\in\mathcal{M}_n(K_2\vee (k-2)K_1)$ if equality holds.

	

	\textbf{Case 2.} $G$ is disconnected.

		Then $G\cong \bigcup\limits_{i=1}^{l}G_i$, where $l\geq2$ and each $G_i$ is  connected.
		By Lemma \ref{l13}, we have $G\cong G_1\cup (l-1) K_5$, where $G_1$  is either a $K^2_4$-saturated graph,    or isomorphic to 
		$K_{f}$ with $1\leq f\leq 5$.
		

	By Case 1 and $e((l-1) K_5)=2(n-|V(G_1)|)$,  we have  $$e(G)\geq\begin{cases}
		2|V(G_1)|-3+2(n-|V(G_1)|)=2n-3,&\text{if}\ G_1 \ \text{is a $K^2_4$-saturated graph}; \\
		\frac{f(f-1)}{2}+2(n-f) \geq 2n-3,&\text{if}\ G_i\cong K_{f}.
	\end{cases}$$
Clearly, if $e(G)=2n-3$, then $ G\in   (\mathcal{M}_{q'}(K_2)\cup \frac{n-q'}{5}K_5)\cup (\mathcal{M}_{q'}(K_3)\cup\frac{n-q'}{5}K_5)\cup( \mathcal{M}_{p'}(K_2\vee (k'-2)K_1)\cup\frac{n-p'}{5}K_5)$ with $q'\leq n-5$, $6\leq p'\leq n-5$ and $4\leq k'\leq p$. 

Clearly, $\mathcal{M}_{n}(K_i)$ can be regarded as $(\mathcal{M}_{q}(K_i)\cup \frac{n-q}{5}K_5)$ with $q=n$, where $i\in \{2,3\}$, $K_2\vee (n-2)K_1$  can be regarded as $\mathcal{M}_{n}(K_2\vee (k-2)K_1)$ with $k=n$,   and $ \mathcal{M}_{n}(K_2\vee (k-2)K_1)$ can be regarded as  $\mathcal{M}_{p}(K_2\vee (k-2)K_1)\cup\frac{n-p}{5}K_5$ with $p=n$, which implies that the proof is completed by Cases 1--2.
 $\hfill\blacksquare$

	
	
	

		\noindent\textbf{\textit{Proof of Theorem \ref{t3}.}} Clearly, if $G\in  \mathcal{SM}^t_n(2K_2\vee K_{t-2})\cup\mathcal{SM}^t_n(K_r)\cup
		\mathcal{SM}^t_n(K_1)$, then $G$ is a $K^2_{t}$-saturated graph with $$e(G)=\begin{cases}
			\frac{tn}{2}-\frac{t+2}{2},&\text{if}\ G\in {SM}^t_n( K_{t-2}\vee 2K_1) \ \text{with}\ n\equiv 1\ (\mathrm{mod}\ t-1)  \ \text{and} \ n\neq2t-1; \\
			\frac{tn}{2}-\frac{r(t-r+1)}{2},&\text{if}\ G\in \mathcal{SM}^t_n(K_r) \ \text{with}\  n\equiv r\ (\mathrm{mod}\ t-1);\\
			t^2-t,&\text{if}\ G\in \mathcal{SM}^t_n(K_1) \ \text{with} \ n=2t-1.
		\end{cases}$$

		Now we show (\ref{s11}) and (\ref{s12}) hold.
		Let $G$ be a connected $K^2_t$-saturated graph of order $n\geq t+2\geq7$, $G'$ be defined as Lemma \ref{l4}. 
	 The we  consider the following two cases.

			\textbf{Case 1.} $G'$ is $K^2_t$-saturated.
			
			 Then  $|V(G')|\geq t+2$ and $G'$ is $K_t$-saturated by Lemma \ref{l4}. By  Lemma \ref{l4}, Theorem \ref{e2} and $|V(G')|\geq t+2$, we have 
		\begin{align}
			e(G)\nonumber&\geq e(G')+\frac{t}{2}|V(G)\setminus V(G')|\\ \nonumber &\geq \frac{(t-2)(t-3)}{2}+(|V(G')|-(t-2))(t-2)+\frac{t}{2}(n-|V(G')|) \\ \nonumber
			&= \frac{tn}{2}+\frac{t-4}{2}|V(G')|-\frac{(t-2)(t-1)}{2} \\ \nonumber
			&\geq \frac{tn}{2}+\frac{t-10}{2} \\ \nonumber &>\begin{cases}
				\frac{tn}{2}-\frac{t+2}{2},&\text{if}\ n\equiv 1\ (\mathrm{mod}\ t-1)  \ \text{and} \ n\neq2t-1; \\
				\frac{tn}{2}-\frac{r(t-r+1)}{2},&\text{if}\ n\equiv r\ (\mathrm{mod}\ t-1)\ \text{with} \ 2\leq r\leq t-1;\\
				t^2-t,&\text{if}\ n=2t-1.
			\end{cases} \nonumber
		\end{align}
		

		
		\textbf{Case 2.} $G'$ is not $K^2_t$-saturated.
		
		\textbf{Subcase 2.1.} $V(G')$ is a clique.
		
		By Lemma \ref{l4}, we have that $G'$ is $K_t$-free, so $G'\cong K_{f}$ with $1\leq f\leq t-1$. By Lemma \ref{l4} and $t\geq5$, we have $e(G)\geq \frac{f(f-1)}{2}+\frac{t(n-f)}{2}=\frac{tn}{2}-\frac{f(t-f+1)}{2}$, and equality holds if and only if  every $K_t$ in $G$ is a type-I   $K_t$, which implies that    $ n\equiv f\pmod  {t-1}$. Specially, if $f=2$ and there exists $v_1\in V(G')$ such that   $p_1= 1$, then $G$ is not $K^2_t$-saturated, a contradiction. Thus $G\in \mathcal{SM}^t_{n}(K_{f})$. Clearly, 	if $f=1$ and $n\neq 2t-1$, then $e(G)=\frac{tn}{2}-\frac{t}{2}>\frac{tn}{2}-\frac{t+2}{2}$; if $f=1$ and $n= 2t-1$, then $G\in \mathcal{SM}^t_n(K_1)$ with $e(G)=t^2-t$.

		
		\textbf{Subcase 2.2.} $V(G')$ is not clique.

		Then there exist two vertices $u,v\in V(G')$ such that $uv\notin E(G')\subseteq E(G)$ and $G'+uv$ contains no copy of $K^2_t$. Since $G$ is $K^2_t$-saturated, $G+uv$ contains a copy of $K^2_t$, say $Q$. Now we show $u,v\in V(CQ)$.  Otherwise,  without loss of generality we may assume that $u\in V(CQ)$ and $d_Q(v)=1$. Then $G[ V(CQ)]\cong K_t$ and $u\in V(K_t)$. Thus $G[ V(CQ)]$ is a type-I  $K_t$ and $u$ is a lucky vertex by Definition \ref{d2} and $u\in V(G')$, which implies that $Q\not\cong K^2_t$, a contradiction. Therefore, we have $u,v\in V(CQ)$. Suppose that $V(CQ)=\{u,v,u_1,u_2,\ldots,u_{t-2}\}$.   
		
		\textbf{Claim 1.} $V(CQ)\subseteq V(G')$. 
		
		\begin{proof}
			Suppose to the contrary that  there exists $i\in [t-2]$ such that $u_i\in V(G)\setminus V(G')$. Then $u_i$ belongs to an unlucky $K_{t-1}$ or a type-II $K_t$, and thus there exists at most one vertex in $N(u_i)$ belonging to $V(G')$, which  contradicts $u,v\in N_G(u_i)\cap V(G')$.
				\end{proof}

			\textbf{Claim 2.} $ |V(G')|\leq t+1$. 
			
			\begin{proof}
					Suppose to the contrary that  $ |V(G')|\geq t+2$. Then   there exists a vertex $w_1\in V(G')\setminus V(CQ)$ such that $w_1\sim  V(CQ)$ since $G'$ is connected. Since $G'+uv$ contains no copy of $K^2_t$, we have one of the following two cases holds: (i)  there exists a unique vertex $v^*\in V(CQ)$ such that $v^*\sim V(G')\setminus V(CQ)$;  (ii) $w\not\sim V(CQ)$ for any $w\in  V(G')\setminus ( V(CQ)\cup\{w_1\})$.
					
					If (i) holds, then $v\not\sim V(G')\setminus V(CQ)$ or $u\not\sim V(G')\setminus V(CQ)$. We may assume that $u\not\sim V(G')\setminus V(CQ)$. Then  $G+uw_1$ contains  a copy of $K^2_t$, say $Q_1$. 
					
				
				If $u,w_1\in V(CQ_1)$, then we claim that $V(CQ_1)\cap V(G')\subseteq \{u,w_1,v^*\}$. Otherwise, there exists $v'\in V(CQ_1)\cap\big(V(G')\setminus\{u,w_1,v^*\}\big)$, and then $v'w_1, v'u\in E(G)$, which contradicts (i) if $v'\in V(CQ)$ or the fact that $u\not\sim V(G')\setminus V(CQ)$ if $v'\in V(G')\setminus V(CQ)$. 
				Thus there exists $w_2\in (V(G)\setminus V(G'))\cap V(CQ_1)$ since $t\geq 5$ and $V(CQ_1)\cap V(G')\subseteq \{u,w_1,v^*\}$, which implies that $u,w_1\in N(w_2)\cap V(G')$, contradicting the fact that at most one vertex in $V(G')$ is adjacent to $w_2$ by the definition of $G'$.
				
				If $d_{Q_1}(w_1)=1$ and $u\in V(CQ_1)$, then $G[V(CQ_1)]\cong K_t$, and thus $G[V(CQ_1)]$ is a type-I  $K_t$ by Proposition \ref{o2}, Definition \ref{d2} and $u\in V(G')$, which implies  only one vertex $u\in V(CQ_1)$ such that $u\sim V(G)\setminus V(CQ_1)$,  contradicting $Q_1\cong K^2_t$.
				
				If $d_{Q_1}(u)=1$ and $w_1\in V(CQ_1)$, then $G[V(CQ_1)]$ is a type-I  $K_t$ by similar arguments, which is also a contradiction.

				If (ii) holds, then $uw\notin E(G)$ and  $G+uw$ contains  a copy of $K^2_t$. By a similar argument, we also reach a contradiction.
				
				Combining the above arguments, Claim $2$ holds.
			\end{proof}

	    By Claim 1 and Claim 2, we have $ t\leq|V(G')|\leq t+1$. Without loss of generality, suppose that $G$ contains $b$ unlucky  $K_{t-1}$ and $c$ type-II   $K_t$, where  $b,c\geq0$ and $b+c\geq1$ since $n\geq t+2$. Now we consider  the following two subcases.
	    
	    	\textbf{Subcase 2.2.1} $|V(G')|=t$.
	    
	    Then $G'\cong K_{t-2}\vee 2K_1$ by Claim 1 and $n=t+b(t-1)+ct$, and thus there exist at least two vertices $v_i,v_j\in V(G')$ such that $p_i+q_i\neq0$ and $p_j+q_j\neq0$ since $G'+uv\cong K_t$ and $G+uv$ contains $Q$, which implies $b+c\geq2$ and $n>2t-1$. Moreover, we have 
	    \begin{align}\label{s1}
	    	e(G)= e(G')+b\frac{t(t-1)}{2}+c(\frac{t(t-1)}{2}+t-2)
	    	= \frac{tn}{2}-\frac{t+2}{2}+\frac{c(t-4)}{2}.  
	    \end{align}

	    If $c=0$,   then $e(G)= \frac{tn}{2}-\frac{t+2}{2}$ and $G\in \mathcal{SM}^t_{n}(K_{t-2}\vee2K_1 )$ with $n\equiv1 \pmod{t-1}$  and $n\neq 2t-1$. 
	    If $1\leq c\leq2$, then $n\equiv c+1\pmod{t-1}$  and  $e(G)>\max \{	\frac{tn}{2}-\frac{t+2}{2}, 	\frac{tn}{2}-\frac{r(t-r+1)}{2}\}$  since (\ref{s1}) and   $r\equiv c+1\pmod{t-1}$. If $c\geq3$, then $e(G)\geq\frac{tn}{2}-\frac{t-1}{2}>\max \{	\frac{tn}{2}-\frac{t+2}{2}, 	\frac{tn}{2}-\frac{r(t-r+1)}{2}\}$ with $2\leq r\leq t-1$ since (\ref{s1}) and $\frac{c(t-4)}{2}\geq \frac{3}{2}$.
	    
	    
	    \textbf{Subcase 2.2.2} $|V(G')|=t+1$.

	    Then $n=t+1+b(t-1)+ct>2t-1$ and there exists $w^*\in V(G')\setminus V(CQ)$ such that $V(G')=V(CQ)\cup \{w^*\}$, so \begin{align}\label{s2}
	    	e(G)\nonumber&= e(G')+b\frac{t(t-1)}{2}+c(\frac{t(t-1)}{2}+t-2)\\  &= (\frac{(t-2)(t-3)}{2}+2(t-2)+d_{G'}(w^*)) +b\frac{t(t-1)}{2}+c(\frac{t(t-1)}{2}+t-2)\nonumber \\ &= \frac{tn}{2}+\frac{c(t-4)}{2}+d_{G'}(w^*)-(t+1). 
	    \end{align}
	    
	    Now we show $d_{G'}(w^*)\geq t-2$. Otherwise, $d_{G'}(w^*)\leq t-3$, then there exists $i\in [t-2]$ such that $w^*u_i\notin  E(G)$. Let $Q_2$ be a copy of $K^2_t$ in $G+w^*u_i$. Then $d_{Q_2}(w^*)=1$, $d_{Q_2}(u_i)=1$, or $w^*,u_i\in V(CQ_2)$. If $d_{Q_2}(w^*)=1$, then $u_i\in V(CQ_2)$ and $G[V(CQ_2)]\cong K_t$, and thus $u_i$ is  a lucky vertex since $u_i\in V(G')$, which  contradicts  $Q_2\cong K^2_t$. Similarly, if $d_{Q_2}(u_i)=1$, we also obtain a contradiction. If $w^*,u_i\in V(CQ_2)$, then there exists $w'\in V(CQ_2)\cap (V(G)\setminus V(G'))$ since $d_{G'}(w^*)\leq t-3$, and thus $w'$ lies  in an unlucky $K_{t-1}$ or a type-II $K_t$, which  contradicts  $w^*,u_i\in N(w')\cap V(G')$. Combining the above arguments, we have $d_{G'}(w^*)\geq t-2$.

	    In the following, we show  $e(G)>\max \{	\frac{tn}{2}-\frac{t+2}{2}, 	\frac{tn}{2}-\frac{r(t-r+1)}{2}\}$ with $2\leq r\leq t-1$.
	    
	    If $d_{G'}(w^*)\geq t-1$, then $e(G)\geq \frac{tn}{2}-2>\max \{	\frac{tn}{2}-\frac{t+2}{2}, 	\frac{tn}{2}-\frac{r(t-r+1)}{2}\}$ since (\ref{s2}) and $t\geq5$. 
	    
	    If $d_{G'}(w^*)= t-2$ and $0\leq c\leq 1$, then  $n\equiv c+2\pmod{t-1}$ and 
	    $e(G)=\frac{tn}{2}+\frac{c(t-4)}{2}-3>\max \{	\frac{tn}{2}-\frac{t+2}{2}, 	\frac{tn}{2}-\frac{r(t-r+1)}{2}\}$ since (\ref{s2}), $r\equiv c+2\pmod{t-1}$  and $t\geq5$. 
	    
	    If $d_{G'}(w^*)= t-2$ and $ c\geq 2$, then $e(G)\geq \frac{tn}{2}-2>\max \{	\frac{tn}{2}-\frac{t+2}{2}, 	\frac{tn}{2}-\frac{r(t-r+1)}{2}\}$ with $2\leq r\leq t-1$ since (\ref{s2}) and $t\geq5$.
	    
	    By Case 1 and Case 2, both (\ref{s11}) and (\ref{s12}) hold. 	    $\hfill\blacksquare$
	    
	    
	    
	    
	  \vspace{0.2cm}
	  
	  Based on the result of Theorem \ref{t3}, we now proceed to the proof of Theorem \ref{t4}.

	  	\noindent\textbf{\textit{Proof of Theorem \ref{t4}.}}	If $t\equiv 0\ (\mathrm{mod}\ 2)$, $H\in \mathcal{SM}^t_{n-\beta (t+1)}(K_{\gamma})\cup \beta K_{t+1}$ with
	  	
	  	$$(\beta,\gamma)=\begin{cases}
	  		(\frac{2f-t}{4}, \frac{t}{2}),&\text{if}\ f\geq \frac{t}{2}+2 \ \text{and}\ f+\frac{t}{2}\equiv 0\ (\mathrm{mod}\ 2); \\
	  		(\frac{2f-t-2}{4}, \frac{t+2}{2}),&\text{if}\ f\geq \frac{t}{2}+2 \ \text{and}\ f+\frac{t}{2}\equiv 1\ (\mathrm{mod}\ 2); \\
	  		(\frac{t}{2}, \frac{t}{2}),&\text{if}\ f= \frac{t}{2}+1; \\
	  		(\frac{t-2}{2}, \frac{t+2}{2}),&\text{if}\ f= \frac{t}{2}; \\
	  		(\frac{2f+t-4}{4}, \frac{t+2}{2}),&\text{if}\ f\leq \frac{t}{2}-1 \ \text{and}\ f+\frac{t}{2}\equiv 0\ (\mathrm{mod}\ 2); \\
	  		(\frac{2f+t-2}{4}, \frac{t}{2}),&\text{if}\ f\leq \frac{t}{2}-1 \ \text{and}\ f+\frac{t}{2}\equiv 1\ (\mathrm{mod}\ 2), \\
	  	\end{cases}$$
	  	then $H$ is a $K^2_t$-saturated graph of order $n$ and 
	  	$e(H)=\frac{tn}{2}-\frac{t(t+2)}{8}$.

	  	If $t\equiv 1\ (\mathrm{mod}\ 2)$, $H\in \mathcal{SM}^t_{n-\beta (t+1)}(K_{\gamma})\cup \beta K_{t+1}$ with
	  	
	  	$$(\beta,\gamma)=\begin{cases}
	  		(\frac{2f-t-1}{4}, \frac{t+1}{2}),&\text{if}\ f\geq \frac{t+1}{2}+1 \ \text{and}\ f+\frac{t+1}{2}\equiv 0\ (\mathrm{mod}\ 2); \\
	  		(\frac{2f-t+1}{4}, \frac{t-1}{2}),&\text{if}\ f\geq \frac{t+1}{2}+1 \ \text{and}\ f+\frac{t+1}{2}\equiv 1\ (\mathrm{mod}\ 2); \\
	  		(\frac{t-1}{2}, \frac{t+1}{2}),&\text{if}\ f= \frac{t+1}{2}; \\
	  		(\frac{2f+t-3}{4}, \frac{t+1}{2}),&\text{if}\ f\leq \frac{t+1}{2}-1 \ \text{and}\ f+\frac{t+1}{2}\equiv 0\ (\mathrm{mod}\ 2); \\
	  		(\frac{2f+t-5}{4}, \frac{t+3}{2}),&\text{if}\ f\leq \frac{t+1}{2}-1 \ \text{and}\ f+\frac{t+1}{2}\equiv 1\ (\mathrm{mod}\ 2), \\
	  	\end{cases}$$
	  	then $H$ is a $K^2_t$-saturated graph of order $n$ and $$e(H)=\begin{cases}
	  		\frac{tn}{2}-\frac{(t+1)^2}{8},&\text{if}\  f+\frac{t+1}{2}\equiv 0\ (\mathrm{mod}\ 2); \\
	  		\frac{tn}{2}-\frac{(t+3)(t-1)}{8},&\text{if}\  f+\frac{t+1}{2}\equiv 1\ (\mathrm{mod}\ 2). \\
	  	\end{cases}$$
	  	
	  	We claim that $H$ is a well-defined graph in each case. We only present the detailed proof for the case  $t\equiv 0\ (\mathrm{mod}\ 2)$, $f\geq \frac{t}{2}+2$ and $ f+\frac{t}{2}\equiv 0\ (\mathrm{mod}\ 2)$, and the remaining cases follow similarly. Clearly, $n-\beta (t+1)=f+k_1(t-1)-\frac{2f-t}{4}(t+1)=\frac{t}{2}+(k_1-\frac{2f-t}{4})(t-1)$ and both $\frac{2f-t}{4}$ and $k_1-\frac{2f-t}{4}$  are positive integers since $f\geq \frac{t}{2}+2 $, $ f+\frac{t}{2}\equiv 0\ (\mathrm{mod}\ 2)$ and $n=f+k_1(t-1)\geq \frac{(t+3)(t-1)+t+1}{2}$. Therefore, $H$ is a well-defined graph if $H\in \mathcal{SM}^t_{n-\beta (t+1)}(K_{\gamma})\cup \beta K_{t+1}$.

	  	Combining the above arguments, we have $\mathrm{sat}(n,K^2_t)\leq g(n,t)$.

	  	Let $G$ be a minimum $K^2_t$-saturated graph of order $n$. Then \begin{equation}\label{s7}
	  		e(G)= \mathrm{sat}(n,K^2_t)\leq g(n,t).
	  	\end{equation} 
	  	
	  	Now we  show $e(G)\geq g(n,t)$
	  	by considering the following two cases. 
	  	
	  	\textbf{Case 1.} $G$ is connected.
	  	
	  	By Theorem \ref{t3} and (\ref{s7}), we have $G\in \mathcal{SM}^t_{n}(K_f)$ with \begin{equation}
	  		f=\begin{cases}
	  			\frac{t}{2} \ \text{or} \ \frac{t+2}{2}, &\text{if} \ t \ \text{is even};\\
	  			\frac{t-1}{2}, \ \frac{t+1}{2},  \ \text{or} \ \frac{t+3}{2}, &\text{if} \ t \ \text{is odd},\nonumber
	  		\end{cases}
	  	\end{equation}
	  	and thus $e(G)\geq g(n,t)$.
	  	
	  	\textbf{Case 2.} $G$ is disconnected.
	  	
	  	Then $G\cong \bigcup\limits_{i=1}^{l}G_i$, where $l\geq2$ and each $G_i$ is  connected. By Lemma \ref{l13},  we have $G\cong G_1\cup (l-1) K_{t+1}$, where $G_1$  is either a $K^2_t$-saturated graph,    or isomorphic to 
	  	$K_{f'}$ with $1\leq f'\leq t+1$.
	  	
	  	\textbf{Subcase 2.1.} $G_1\cong K_{f'}$ with $1\leq f'\leq t+1$.
	  	
	  	Clearly, $e(G)=\frac{tn}{2}-\frac{f'(t-f'+1)}{2}$. By (\ref{s7}), we have  \begin{equation}
	  		f'=\begin{cases}
	  			\frac{t}{2} \ \text{or} \ \frac{t+2}{2}, &\text{if} \ t \ \text{is even};\\
	  			\frac{t-1}{2}, \ \frac{t+1}{2},  \ \text{or} \ \frac{t+3}{2}, &\text{if} \ t \ \text{is odd},\nonumber
	  		\end{cases}
	  	\end{equation}
	  	and thus $e(G)\geq g(n,t)$.
	  	
	  	\textbf{Subcase 2.2. } $G_1$  is $K^2_t$-saturated.
	  	
	  	Let $|V(G_1)|\equiv r'\ (mod \ t-1)$ with $1\leq r'\leq t-1$. Then 
	  	
	  	\begin{equation}
	  		e(G_1)\geq\begin{cases}
	  			\frac{t|V(G_1)|}{2}-\frac{t+2}{2},&\text{if}\ r'=1 \ \text{and} \ |V(G_1)|\neq2t-1; \\
	  			\frac{t|V(G_1)|}{2}-\frac{r'(t-r'+1)}{2},&\text{if}\  2\leq r'\leq t-1;\\
	  			t^2-t,&\text{if} \ |V(G_1)|=2t-1\nonumber
	  		\end{cases}
	  	\end{equation} 
	  	by Theorem \ref{t3}. 
	  	Thus we have 
	  	\begin{equation}\label{s5}
	  		e(G)\geq\begin{cases}
	  			\frac{tn}{2}-\frac{t+2}{2},&\text{if}\ r'=1 \ \text{and} \ |V(G_1)|\neq2t-1; \\
	  			\frac{tn}{2}-\frac{r'(t-r'+1)}{2},&\text{if}\  2\leq r'\leq t-1;\\
	  			\frac{tn}{2}-\frac{t}{2},&\text{if} \ |V(G_1)|=2t-1,
	  		\end{cases}
	  	\end{equation}
	  	
	  	and equality holds if and only if \begin{equation}\label{s6}
	  		G\in \begin{cases}
	  			\mathcal{SM}^t_{|V(G_1)|}(2K_1\vee K_{t-2})\cup \frac{n-|V(G_1)|}{t+1}K_{t+1},\ \text{if}\ r'=1 \ \text{and} \ |V(G_1)|\neq2t-1; \\
	  			\mathcal{SM}^t_{|V(G_1)|}(K_{r'})\cup \frac{n-|V(G_1)|}{t+1}K_{t+1}, \ \text{if}\ 2\leq r'\leq t-1;\\
	  			\mathcal{SM}^t_{|V(G_1)|}(K_1)\cup \frac{n-|V(G_1)|}{t+1}K_{t+1}, \ \text{if}\ |V(G_1)|=2t-1
	  		\end{cases}
	  	\end{equation}
	  	by Theorem \ref{t3}.

	  	By (\ref{s7})--(\ref{s6}) and $t\geq5$, we have $G\in \mathcal{SM}^t_{|V(G_1)|}(K_{r'})\cup \frac{n-|V(G_1)|}{t+1}K_{t+1}$ with $\frac{t-1}{2}\leq r'\leq \frac{t+3}{2}$.

	  	If $t$ is even, then $r'\neq \frac{t+1}{2}$, and thus $e(G)\geq \frac{tn}{2}-\frac{t(t+2)}{8}$. 
	  	
	  	If $t$ is odd and $f$ has the same parity as $\frac{t+1}{2}$, then $e(G)\geq \frac{tn}{2}-\frac{(t+1)^2}{8}$.
	  	
	  	If $t$ is odd and $f$  has different parity from $\frac{t+1}{2}$, then $r'\neq \frac{t+1}{2}$, and thus $e(G)\geq \frac{tn}{2}-\frac{(t+3)(t-1)}{8}$. 
	  	
	  	Combining the above arguments,	we complete the proof.    $\hfill\blacksquare$

	\section{The saturation number of $K^3_3$}

	\begin{lemma}\label{l1}
		Let $G$ be a connected $K^3_3$-saturated graph. Then $2\leq diam(G)\leq 4$.
	\end{lemma}
	
	 \begin{proof}
	 	Clearly, we have $ diam(G)\geq2$. Next, we show $diam(G)\leq 4$.
	 	
	 	Suppose to the contrary that $diam(G)\geq 5$. Then there exists a path $P=v_1v_2v_3v_4v_5v_6$  of $G$ with $d(v_1,v_6)=5$.
	 	
	 	\textbf{Claim A.} Let $v_i,v_j\in V(P)$, $d(v_i,v_j)\geq3$, $Q$ be a copy of  $K^3_3$ in $G+v_iv_j$, $d_Q(v_i)=1$, $v_j\in V(CQ)$ with $d_P(v_j)=2$. Then there exists $u\in V(G)\setminus V(P)$ such that $uv_{j-1},uv_{j},uv_{j+1}\in E(G)$ and $d(u)=d(v_j)=3$.
	 	
	 	\begin{proof}
	 		Since $d(v_1,v_6)=5$, we have $V(CQ)\cap V(P)=\{v_j\}$ or $\{v_j, v_p\}$, where  $p\in \{j-1,j+1\}$.
	 		
	 		If $V(CQ)\cap V(P)=\{v_j\}$, then there exist $u_1,u_2\in V(G)\setminus V(P)$ such that $V(CQ)=\{v_j, u_1,u_2\}$. Moreover, we have $v_{j-1}, v_{j+1}\in V(Q)$. Otherwise, there exists $k\in \{j-1,j+1\}$ such that $v_k\notin V(Q)$, then  $G[(V(Q)\setminus\{v_i\})\cup \{v_k\}]$ contains  a copy of $K^3_3$, a contradiction. Based on the fact $v_{j-1}, v_{j+1}\in V(Q)$, we have  $G[\{u_1,u_2,v_j,v_{j-1},v_{j+1},v_{j-2}\}]$ contains  a copy of $K^3_3$ (if $j-i\geq3$), or $G[\{u_1,u_2,v_j,v_{j-1},v_{j+1},v_{j+2}\}]$ contains  a copy of $K^3_3$ (if $i-j\geq3$), a contradiction.
	 		
	 		If $V(CQ)\cap V(P)=\{v_j, v_p\}$, then there exists $u\in V(G)\setminus V(P)$ such that $V(CQ)=\{v_j, v_p,u\}$. If $uv_{q}\notin E(G)$ for $q\in \{j-1,j+1\}\setminus \{p\}$, then $v_q\notin V(Q)$, which implies that $G[(V(Q)\setminus\{v_i\})\cup \{v_q\}]$ contains  a copy of $K^3_3$, a contradiction. Therefore, $uv_{q}\in E(G)$, and $uv_{j-1},uv_{j},uv_{j+1}\in E(G)$ holds since $\{p,q\}=\{j-1,j+1\}$.

	 		 Clearly, we have $d(u)=d(v_j)=3$. Otherwise, $G$ contains  a copy of $K^3_3$, a contradiction.
	 		We complete the proof of the Claim A.
	 	\end{proof}

	 	Let $Q_1$ be a copy of $K^3_3$ in $G+v_2v_5$. Then $d_{Q_1}(v_2)=1$ or $d_{Q_1}(v_5)=1$ since $d(v_2,v_5)=3$. Without loss of generality, let $d_{Q_1}(v_5)=1$. Then $v_2\in V(CQ_1)$ with $d_P(v_2)=2$, and thus there exists $u_3\in V(G)\setminus V(P)$ such that $u_3v_1,u_3v_2,u_3v_3\in E(G)$ and $d(u_3)=d(v_2)=3$ by Claim A.
	 	
	 	 Let $Q_2$ be a copy of $K^3_3$ in $G+v_1v_4$. We claim  $d_{Q_2}(v_1)=1$. Otherwise, $v_1\in V(CQ_2)$. If $V(CQ_2)\cap \{v_1,v_2,u_3\}=\{v_1\}$,   then $G[(V(Q_2)\setminus\{v_4\})\cup \{v_2\}]$ contains  a copy of $K^3_3$ since $v_1v_4\in E(Q_2)$ and $v_2\notin V(Q_2)$ with $N(v_2)=\{v_1,v_3,u_3\}$, a contradiction. If $V(CQ_2)\cap \{v_1,v_2,u_3\}=\{v_1,v_2\}$, $\{v_1,u_3\}$ or $\{v_1,v_2, u_3\}$, then $d(v_2)\geq 4$ or $d(u_3)\geq 4$, a contradiction. Then $d_{Q_2}(v_1)=1$ and $v_4\in V(CQ_2)$ with $d_P(v_4)=2$, and thus there exists $u_4\in V(G)\setminus (V(P)\cup\{u_3\})$ such that $u_4v_3,u_4v_4,u_4v_5\in E(G)$ and $d(u_4)=d(v_4)=3$ by Claim A.

	 	  Let $Q_3$ be a copy of $K^3_3$ in $G+v_1v_5$. Similar the proof of $G+v_1v_4$, we have $d_{Q_3}(v_1)=1$ and $v_5\in V(CQ_3)$ with $d_P(v_5)=2$.
	 	   Then there exists  $u_5\in V(G)\setminus (V(P)\cup\{u_3,u_4\})$ such that $u_5v_4,u_5v_5,u_5v_6\in E(G)$ by Claim A. 
	 	 
	 		Combining the above arguments, we have $\{v_3,v_5,u_4,u_5\}\subseteq N(v_4)$ and $d(v_4)\geq4$, which  contradicts $d(v_4)=3$. Therefore,  $2\leq diam(G)\leq 4$.
	 \end{proof}
	 
	 We note that a connected $K^3_3$-free graph of order at least 6 is also $K_1\vee P_4$-free. Then we immediately have the following conclusion.
	 
	 \begin{Observation}\label{o}
	 	Let $G$ be a connected $K^3_3$-saturated graph of order $n\geq6$. Then $P_4$ is not a subgraph of $G[N(v)]$ for any vertex $v\in V(G)$.
	 \end{Observation}
	
Let $G$ be a graph, $U\subseteq V(G)$, we use $G-U $ to denote a new graph from $G$ by removing   all vertices in $U$  from  $G$. For any non-empty vertex set $N_0\subset V(G)$, we denote $N_1=\{v\in V(G)\mid v\sim N_0 \}\setminus N_0$, $N_i=\{v\in V(G)\mid v\sim N_{i-1} \}\setminus (N_{i-1}\cup N_{i-2})$, where $i\geq 2$.

\begin{theorem}\label{l2}
	Let $G$ be a connected $K^3_3$-saturated graph of order $n\geq6$, $\delta(G)=1$. Then $G\cong K_1\vee(K_1\cup \frac{n-2}{3}K_3)$ with $n\equiv  2 \ (mod \ 3)$, and $e(G)=2n-3> \frac{3n}{2}$ with $n\geq8$.
\end{theorem}

\begin{proof}
	Let $v\in V(G)$ such that $d(v)=\delta(G)=1$, $N(v)=\{v_1\}$, $N_0=\{v\}$. Then  $V(G)=\bigcup\limits_{i=0}^{\varepsilon(v)}N_i$, and $2\leq \varepsilon(v)\leq 4$ by Lemma \ref{l1}. Firstly, we prove $\varepsilon(v)=2$.
	
	 	Suppose to the contrary that $\varepsilon(v)\geq 3$. Then there exists $v_2\in N_2, u_1\in N_3$ such that $v_2u_1\in E(G)$. Let $Q_1$ be a copy of $K^3_3$ in $G+vv_2$. Then $d_{Q_1}(v)=1$ and $v_2\in V(CQ_1)$.
	 	
	 	\textbf{Case 1.} $v_1\notin V(CQ_1)$.
	 	
	 	Suppose that $V(CQ_1)=\{v_2,v_3,v_4\}$. Then $v_3,v_4\in N_2\cup N_3$.
	 	
	 	 		If $v_3,v_4\in N_3$, then $G[(V(Q_2)\setminus\{v\})\cup \{v_1\}]$ contains  a copy of $K^3_3$, a contradiction.
	 	 		
	 	 		 	If $v_3,v_4\in N_2$, then $G[\{v_1,v_2,v_4,v,u_1,v_3\}]$ contains  a copy of $K^3_3$, a contradiction.
	 	
	 	If $v_3\in N_2$ and $v_4\in N_3$, then $d(v_2)=d(v_3)=3$ since $G$ is $K^3_3$-free. Otherwise, $G$ contains  a copy of $K^3_3$, a contradiction. Let $Q_2$ be a copy of $K^3_3$ in $G+vv_4$. Then $d_{Q_2}(v)=1$, $v_2,v_3\notin V(CQ_2)$, and thus $v_2,v_3\notin V(Q_2)$, which implies that $G[(V(Q_2)\setminus \{v\})\cup \{v_2\}]$ contains  a copy of $K^3_3$, a contradiction.

	 		\textbf{Case 2.} $v_1\in V(CQ_1)$.
	 		
	 	Suppose that $V(CQ_1)=\{v_1,v_2,v_5\}$, then $v_5\in N_2$ and  there exists $v_6\in V(Q_1)\setminus V(CQ_1)$ such that $v_6v_5\in E(Q_1)$.
	 		
	 		If $v_6\in N_2$ or $v_6\in N_3$ with $v_6v_2\notin E(G)$, then  $G[(V(Q_1)\setminus\{v\})\cup \{u_1\}]$ contains  a copy of $K^3_3$, a contradiction.
	 		
	 		If $v_6\in N_3$ with $v_6v_2\in E(G)$, then $d(v_2)=d(v_5)=3$ and $v_6$ is $u_1$.  Let $Q_3$ be a copy of $K^3_3$ in $G+vv_6$. Then $d_{Q_3}(v)=1$ and $v_6\in V(CQ_3)$, and thus $v_2,v_5\notin V(CQ_3)$ since $d(v_2)=d(v_5)=3$ and $N(v_2)\setminus\{v_5\}=N(v_5)\setminus\{v_2\}$, which implies that $v_2,v_5\notin V(Q_3)$. Therefore, $G[(V(Q_3)\setminus\{v\})\cup \{v_2\}]$ contains  a copy of $K^3_3$, a contradiction.

	 		Combining the above arguments, we have $\varepsilon(v)=2$.

	 		 By $\varepsilon(v)=2$ and $d(v)=1$, we have $d(v_1)=n-1$. Then each connected component of $G-v_1$ is $K_1$, $K_{1,r}$ with $r\geq1$, or $K_3$ by Observation \ref{o}. If some connected component $H_1$ of $G-\{v_1\}$ is $K_{1,r}$, then $G+vu$ contains no  copy of $K^3_3$, where $d_{H_1}(u)=r$, a contradiction. Similarly, there exists exactly one
	 		 connected component $H_2$ of $G-\{v_1\}$ that is $K_1$, where $V(H_2)=\{v\}$. Therefore, 
	 		 $G\cong K_1\vee (K_1\cup \frac{n-2}{3}K_3)$ with $n\equiv  2 \ (mod \ 3)$, and  $e(G)=2n-3>\frac{3n}{2}$ for $n\geq8$.
\end{proof}

 Now, we define four important types of vertices of degree $2$.

	\begin{definition}\label{d1}
		Let $v\in V(G)$ with $d(v)=2$ and $N(v)=\{v_1,v_2\}$.
		
		{\rm \item(i)} If $d(v_1)=2$, there exists $v_3\in N(v_1)\setminus \{v\}$ such that $v_3v_2\in E(G)$,  $d(v_2)\geq3$ and $d(v_3)\geq3$, then $v$ is called  a type-I vertex;
		
		{\rm \item(ii)} If $d(v_1)=2$, $N(v_1)=\{v, v_2\}$, and $d(v_2)\geq3$,  then $v$ is called  a type-II vertex;
		
		{\rm \item(iii)} If $d(v_1)\geq3$,   $d(v_2)\geq3$, and $v_1v_2\notin E(G)$,  then $v$ is called  a type-III vertex;
		
		{\rm \item(iv)} If $d(v_1)\geq3$,   $d(v_2)\geq3$,  $v_1v_2\in E(G)$, and there exist two distinct vertices  $v'_1,v'_2\in V(G)$ such that $v_1v'_1,v_2v'_2\in E(G)$, then $v$ is called  a type-IV vertex.
	\end{definition}
	
	\begin{lemma}\label{ll1}
		Let $G$ be a connected $K^3_3$-saturated graph of order $n\geq6$, $v\in V(G)$ with $d(v)=\delta(G)=2$. Then $v$ is  a type-I, type-II, type-III, or type-IV vertex.
	\end{lemma}
	
	\begin{proof}
		Without loss of generality, let $N(v)=\{v_1,v_2\}$. 
		
		First, we show $d(v_1)\geq3$ or $d(v_2)\geq3$. Otherwise, $d(v_1)=d(v_2)=2$ since $\delta(G)=2$, then $v_1v_2\notin E(G)$ since $G$ is connected with order $n\geq6$, but $G+v_1v_2$ contains no copy of $K^3_3$ since $d(v_1)=d(v_2)=2$, a contradiction. Then $d(v_1)\geq3$ or $d(v_2)\geq3$.
		
		 We may assume $d(v_2)\geq3$. Now we complete the proof by the following arguments.
			
			If $d(v_1)=2$ and $v_1v_2\notin E(G)$, then there exists $v_3\in N(v_1)\setminus \{v\}$ such that $v_2v_3\in E(G)$. Otherwise, $G+v_1v_2$ contains no copy of $K^3_3$, a contradiction. Clearly, $d(v_2)\geq3$, $d(v_3)\geq3$, and $v$ is  a type-I vertex.
			
				If $d(v_1)=2$ and $v_1v_2\in E(G)$, then $d(v_2)\geq3$, and  $v$ is  a type-II vertex.
				
				If $d(v_1)\geq3$ and $v_1v_2\notin E(G)$, then $v$ is  a type-III vertex.
		
		 If $d(v_1)\geq3$ and $v_1v_2\in E(G)$, then there must exist two distinct vertices  $v'_1,v'_2$ such that $v_1v'_1,v_2v'_2\in E(G)$. Otherwise,  $d(v_1)=d(v_2)=3$, $v'_1=v'_2$, and  $G+vv'_1$ contains no copy of $K^3_3$, a contradiction. Thus $v'_1\neq v'_2$ and $v$ is  a type-IV vertex.
	\end{proof}

	\begin{lemma}\label{l3}
		Let $G$ be a connected $K^3_3$-saturated graph of order $n\geq6$, $\delta(G)\geq2$, $N_0\subset V(G)$ be a non-empty vertex set such that $V(G)=\bigcup\limits_{i=0}^{d}N_i$ with $3\leq d\leq4$, and there exists a vertex $v\in N_0$ satisfying $v\notin V(CK^2_3)$ for any  subgraph $K^2_3$ of $G$. 
		
		{\rm \item(i)} If $u\in \bigcup\limits_{i=3}^{d}N_i$, then for any $w\in N(u)$, we have $w\sim N(u)\setminus\{w\}$.
		
		{\rm \item(ii)} If $u\in \bigcup\limits_{i=3}^{d}N_i$ with $d(u)=2$, then $u\in N_3$ is a type-IV vertex and $N(u)\subseteq N_2$. 
	\end{lemma}
	
	\begin{proof}
			We show {\rm(i)} and {\rm(ii)} in sequence.
			
			{\rm(i)} Suppose to the contrary that $w\not\sim N(u)\setminus\{w\}$. Let $Q$ be a copy of $K^3_3$ in $G+vu$. Then $d_{Q}(v)=1$ and  $u\in V(CQ)$ since $uv\in E(Q)$, $ d(u,v)\geq i\geq3$, and $v\notin V(CK^2_3)$ for any  subgraph $K^2_3$ of $G$, and thus $w\notin V(Q)$ since $w\not\sim N(u)\setminus\{w\}$, which implies that $G[(V(Q)\setminus\{v\})\cup\{w\}]$ contains  a copy of $K^3_3$, a contradiction.

	{\rm(ii)}	Let $u\in \bigcup\limits_{i=3}^{d}N_i$ with $d(u)=2$. Then there exist $u_1,u_2\in V(G)$ such that $N(u)=\{u_1,u_2\}$. First, we prove that the following claim holds.
		
		\textbf{Claim B.} If $u\in N_j$ with $j\in \{3,4\}$, then $u$ is a type-IV vertex and $\{u_1,u_2\}\subseteq N_{j-1}$.
		
		\begin{proof}
				Let $Q_1$ be a copy of $K^3_3$ in $G+vu$. Then $d_{Q_1}(v)=1$, $u\in V(CQ_1)$ since $d(v,u)\geq j\geq3$ and $v\notin V(CK^2_3)$ for any  subgraph $K^2_3$ of $G$.
				Clearly, $u$ is a type-IV vertex,  $u_1,u_2\in \bigcup\limits_{i=j-1}^{d}N_i$, and $\{u_1,u_2\}\cap N_{j-1}\neq \emptyset$ since $u\in N_j$ and $N(u)=\{u_1,u_2\}$. 
				
					
					If $u_1\in  N_{j-1}$ and $u_2\in  N_j$, then there exists $u'_1\in N_{j-2}$  such that $u'_1u_1\in E(G)$ and $d(u_2,v)\geq j\geq3$. Let $Q_2$ be a copy of $K^3_3$ in $G+vu_2$. Then  $d_{Q_2}(v)=1$ and $u_2\in V(CQ_2)$. It is obvious that $u\notin V(CQ_2)$ since $d(u)=2$. If $u_1\notin  V(CQ_2)$, then $u\notin V(Q_2)$ since $N(u)=\{u_1,u_2\}$, which implies that $G[(V(Q_2)\setminus\{v\})\cup \{u\}]$  contains  a copy of $K^3_3$, a contradiction. Therefore, $V(CQ_2)\cap \{u,u_1,u_2\}=\{u_1,u_2\}$. 
					
					Suppose that  $V(CQ_2)=\{u_1,u_2,u_4\}$, where $u_4\in N_j\cup N_{j-1}$ since $u_1\in  N_{j-1}$, $u_2\in  N_j$ and $u_1u_4,u_2u_4\in E(G)$. If $u_4\in N_j$, then there exists $u_5\in V(G)\setminus((\bigcup\limits_{i=0}^{j-2} N_i)\cup\{u,u_1,u_2,u_4\})$ with $u_5u_4\in E(G)$ since $u_4\in V(CQ_2)$, thus $G[\{u_1,u_2,u_4,u'_1,u,u_5\}]$  contains  a copy of $K^3_3$, a contradiction. If $u_4\in N_{j-1}$, then there exists $u'_4\in N_{j-2}$ such that $u'_4u_4\in E(G)$. If $u'_4\neq u'_1$, then $G[\{u_1,u_2,u_4,u'_1,u,u'_4\}]$ contains  a copy of $K^3_3$, a contradiction. If $u'_4= u'_1$, then there exists $u''_4\in N_{j-3}$ such that $u''_4u'_4\in E(G)$, and thus $G[\{u_1,u_4,u'_4,u,u_2,u''_4\}]$ contains  a copy of $K^3_3$, a contradiction.
					
					Therefore, $u_1,u_2\in N_{j-1}$, and $u$ is a type-IV vertex. Claim B is proved.
		\end{proof}

		 If $u\in N_4$, then $d=4$, and $u$ is a type-IV vertex and $N(u)\subset N_3$ by Claim B.  Clearly, there exist two distinct vertices $u_5,u_6\in V(G)\setminus \{u,u_1,u_2\}$ with $u_1u_5,u_2u_6\in E(G)$ by (iv) of Definition \ref{d1}. Let $Q_3$ be a copy of $K^3_3$ in $G+vu_2$. Then $d_{Q_3}(v)=1$ and $u_2\in V(CQ_3)$, and thus $u_1\in V(CQ_3)$. Otherwise, $u\notin V(Q_3)$, and thus $G[(V(Q_3)\setminus\{v\})\cup \{u\}]$ contains  a copy of $K^3_3$, a contradiction. 
		 
		 Suppose that $V(CQ_3)=\{u_1,u_2,u^*\}$, $u_7\in V(Q_3)\setminus V(CQ_3)$ with $u_7u^*\in E(Q_3)$. Clearly, we have $u^*\in N_2\cup N_3\cup N_4$.  If $u^*\notin \{u_5,u_6\}$, then $G[\{u_1,u_2,u^*,u,u_i,u_7\}]$ contains  a copy of $K^3_3$, where $i=5$ or $6$, a contradiction. If $u^*\in \{u_5,u_6\}$, we may assume $u^*=u_5$. If $u^*\in N_2$, then there exists $u_8\in N_1$ such that $u^*u_8\in E(G)$, and thus $G[\{u_1,u_2,u^*,u,u_6,u_8\}]$ contains  a copy of $K^3_3$, a contradiction. If $u^*\in N_3$ and $u_6\neq u_7$, then $G[\{u_1,u_2,u^*,u,u_6,u_7\}]$ contains  a copy of $K^3_3$, a contradiction. If $u^*\in N_3$ and $u_6= u_7$, then $u_6u^*, u_6u_2\in E(G)$. When $u_6\in N_2$, there exists $u_9\in N_1$ such that $u_9u_6\in E(G)$, and thus  $G[\{u_6,u_2,u^*,u_9,u,u_1\}]$ contains  a copy of $K^3_3$, a contradiction; when $u_6\notin N_2$, there exists $u_{10}\in N_2$ such that $u_2u_{10}\in E(G)$ since $u_2\in N_3$, and thus $G[\{u_1,u_2,u^*,u,u_{10},u_6\}]$ contains  a copy of $K^3_3$, a contradiction.
		 If $u^*\in N_4$, then $u_7\in N_3\cup N_4$, and there exists  $u'_1\in N_2$ such that $u'_1u_1\in E(G)$, which implies that $G[\{u_1,u_2,u^*,u'_1,u,u_7\}]$ contains  a copy of $K^3_3$, a contradiction.

Therefore, $u\in N_3$ is a type-IV vertex and $N(u)\subseteq N_{2}$ by Claim $B$.
	\end{proof}

			\begin{theorem}\label{p1}
			Let $G$ be a connected $K^3_3$-saturated graph of order $n\geq7$, $\delta(G)=2$. Then  $e(G)\geq\begin{cases}
				\frac{3n-3}{2},&\text{if}\ n \ \text{is odd}; \\
				\frac{3n}{2},&\text{if}\ n\ \text{is even}.
			\end{cases}$ Moreover, $e(G)=\frac{3n-3}{2}$ if and only if $G\cong K_1\vee \frac{n-1}{2}K_2$ with odd $n$, and $e(G)=\frac{3n}{2}$ if $G\cong K_1\vee (\frac{n-4}{2}K_2\cup K_3)$ with even $n$.
		\end{theorem}
		
		\begin{proof}
			By $\delta(G)=2$, there exist	 $v,v_1,v_2\in V(G)$  such that  $N(v)=\{v_1,v_2\}$. Thus $v$ is  a type-I, type-II, type-III, or type-IV vertex by Lemma \ref{ll1}.
				
					\textbf{Case 1.} $v$ is  a type-I vertex.
				
				Without loss of generality, we may assume $d(v_1)=2$ and $d(v_2)\geq3$. Then there exists $v_3\in N(v_1)\setminus \{v\}$ such that $d(v_3)\geq3$ and $v_2v_3\in E(G)$ since $v$ is a type-I vertex. Let $N_0=\{v,v_1\}$, $B=\{x\in N_2\mid xv_2, xv_3\in E(G)\}$. Then  $N_1=\{v_2,v_3\}$ and $d(x)=2$ for each $x\in B$ since $G$ is $K^3_3$-free.
				Clearly, $v$ and $v_1$ do not  belong to $V(CK^2_3)$ for any copy of $K^2_3$ in  $G$.

				Let $V(G)=\bigcup\limits_{i=0}^{d}N_i$. Then $d\leq diam(G)$ and  $2\leq d\leq 4$ by Lemma \ref{l1}.
				
				\textbf{Claim 1.} If $N_2\setminus B\neq \emptyset$, then $d(u)\geq3$ for each $u\in N_2\setminus B$.
				
				\begin{proof}
					Suppose to the contrary that there exists $u_1\in N_2\setminus B$ such that $d(u_1)=2$.  	Without loss of generality, we may assume $u_1v_2\in E(G)$ and $u_1v_3\notin E(G)$.
					Let $Q_1$ be a copy of $K^3_3$ in $G+v_1u_1$.
					Then $d_{Q_1}(v_1)=1$, $u_1, v_2\in V(CQ_1)$ since $d(u_1)=2$, which implies that there exists $u_2\in N_2\setminus B$ such that $u_2u_1,u_2v_2\in E(G)$. Similarly, 	let $Q_2$ be a copy of $K^3_3$ in $G+v_1u_2$. Then $d_{Q_2}(v_1)=1$ and $u_2\in V(CQ_2)$, and thus $u_1\notin V(Q_2)$ since $d(u_1)=2$, which implies that $G[(V(Q_2)\setminus \{v_1\})\cup \{u_1\}]$ contains  a copy of $K^3_3$, a contradiction.
				\end{proof}
				
				\textbf{Claim 2.} If $d\geq3$, then $d(w)\geq 3$ for each $w\in \bigcup\limits_{i=3}^{d}N_i$.

				\begin{proof}
					Suppose to the contrary that there exists $w_1\in \bigcup\limits_{i=3}^{d}N_i$ such that $d(w_1)=2$.  Let $N(w_1)=\{w_2,w_3\}$. Then $w_1\in N_3$ is a type-IV vertex and $N(w_1)\subseteq N_2$ by (ii) of Lemma \ref{l3}, and thus $w_2w_3\in E(G)$ by (iv) of Definition \ref{d1}, which implies that $\{w_2,w_3\}\subseteq N_2\setminus B$. 
					
					If $w_2v_i,w_3v_i\in E(G)$, where $i=2$ or $3$, then $d(w_2)=d(w_3)=3$ since $G$ is $K_3^3$-free, and thus $G+v_1w_1$ contains no  copy of $K^3_3$, a contradiction. 
					
					If $w_2v_2,w_3v_3\in E(G)$, then $w_2v_3, w_3v_2\notin E(G)$ since  $G$ is $K_3^3$-free. Let $Q_3$ be  a copy of $K^3_3$ in $G+v_1w_2$. Then $d_{Q_3}(v_1)=1$ and $w_2\in V(CQ_3)$.
					When $V(CQ_3)\cap \{w_2,w_3\}=\{w_2\}$, then $w_1\notin V(Q_3)$ since $N(w_1)=\{w_2,w_3\}$, and $G[(V(Q_3)\setminus \{v_1\})\cup \{w_1\}]$ contains  a copy of $K^3_3$, a contradiction. When $V(CQ_3)\cap \{w_2,w_3\}=\{w_2, w_3\}$, it is clear that $G$ contains  a copy of $K^3_3$, a contradiction.
					
					Combining the above arguments, Claim $2$ holds.
				\end{proof}
				
				By Claim 1 and Claim 2, we have $$e(G)\geq\begin{cases}
					\frac{4+(4+2|B|+|N_2\setminus B|)+(2|B|+3|N_2\setminus B|)}{2}=2n-4,&\text{if}\ d=2; \\
					 \frac{4+(4+2|B|+|N_2\setminus B|)+(2|B|+3|N_2\setminus B|)+3(n-4-|N_2|)}{2}=\frac{3n+|N_2|-4}{2},&\text{if}\ d\geq3.
				\end{cases}$$ 
				
				 Clearly, $2n-4\geq\frac{3n-1}{2}$, and $|N_2|\geq 2$ since $n\geq7$ and $d(v_2)\geq3, d(v_3)\geq3$. 
				 
				 Furthermore, we claim that $ |N_2|\geq3$. Otherwise, suppose that $N_2=\{v',v''\}$. If $N_2=B$, then $N(v')=N(v'')=\{v_2,v_3\}$ and $n=6$, a contradiction. If $|N_2\cap B|\leq1$, we assume that $v'\notin B$, and $v'v_2\in E(G),v'v_3\notin E(G)$, then $v''v_3\in E(G)$ since $d(v_3)\geq3$.  
				Let $Q$ be a copy of $K^3_3$ in $G+v_1v'$. Then $d_Q(v_1)=1$ and $v'\in V(CQ)$. Since $G$ is $K^3_3$-free,  $v_2v''\notin E(G)$ or $v'v''\notin E(G)$. Thus $v_2\notin V(CQ)$. Furthermore, $v_2\notin V(Q)\setminus V(CQ)$, otherwise, $v''\in V(CQ)$ and $v''v_2\in E(Q)$ since $N_2=\{v',v''\}$ and $v'v_3\notin E(G)$, which implies that $v'v''\in E(G)$, a contradiction. So $G[(V(Q)\setminus\{v_1\})\cup \{v_2\}]$ contains a  copy of $K^3_3$, a contradiction. Therefore, we have $|N_2|\geq3$ and $e(G)\geq\frac{3n-1}{2}$.
				

				\textbf{Case 2.} $v$ is  a type-II vertex.
				
				Let $N(v_1)=\{v,v_2\}$, $N_0=\{v,v_1\}$. Then $N_1=\{v_2\}$ and  $V(G)=\bigcup\limits_{i=0}^{\varepsilon(v)}N_i$, where  $2\leq\varepsilon(v)\leq 4$ by Lemma \ref{l1}. Clearly, $v$ and $v_1$ do not  belong to $V(CK^2_3)$ for any copy of $K^2_3$ in  $G$.
				
				\textbf{Claim 3.} If $\varepsilon(v)\geq3$, then $d(w)\geq 3$ for each $w\in \bigcup\limits_{i=3}^{\varepsilon(v)}N_i$.
				
				\begin{proof}
					Suppose to the contrary that there exists $w_1\in \bigcup\limits_{i=3}^{\varepsilon(v)}N_i$ such that $d(w_1)=2$.  Let $N(w_1)=\{w_2,w_3\}$. Then $w_1\in N_3$, $N(w_1)\subseteq N_2$, and $w_2w_3\in E(G)$ by (ii) of Lemma \ref{l3} and (iv) of Definition \ref{d1}, and thus $d(w_2)=d(w_3)=3$ since $w_2v_2,w_3v_2\in E(G)$ and $G$ is $K_3^3$-free. Clearly, $G+vw_1$ contains no  copy of $K^3_3$, a contradiction. 
				\end{proof}
				
				Now we complete the proof by considering $\varepsilon(v)=2$ and $\varepsilon(v)\geq3$.
				
				If $\varepsilon(v)=2$, then $d(v_2)=n-1$, and $e(G)\geq \frac{4+(n-1)+2(n-3)}{2}=\frac{3n-3}{2}$ since $\delta(G)=2$, with equality holds if and only if $d(u)=2$ for each $u\in N_2$. It is clear that $G\cong K_1\vee \frac{n-1}{2}K_2$ if $e(G)=\frac{3n-3}{2}$. If $e(G)\neq \frac{3n-3}{2}$, then $e(G)\geq \frac{3n-2}{2}$, and thus there exist $v_3,v_4,v_5\in N_2$ such that $d(v_3)\geq3$ and $v_4,v_5\in N(v_3)$. If $d(v_3)=3$, then $v_4v_5\in E(G)$. Otherwise, $v_4v_5\notin E(G)$ and $d(v_4)=d(v_5)=2$ by Observation \ref{o}, then $G+v_4v_5$ contains no  copy of $K^3_3$, a contradiction, which implies that $e(G)\geq \frac{3n}{2}$. If $d(v_3)\geq4$, then $e(G)\geq\frac{4+(n-1)+2(n-2)}{2}= \frac{3n-1}{2}$.

				If $\varepsilon(v)\geq 3$, then there exist $v_6\in N_2$, $v_7\in N_3$ such that $v_2v_6,v_6v_7\in E(G)$. Let $Q_4$ be  a copy of $K^3_3$ in $G+vv_7$. Then $d_{Q_4}(v)=1$ and $v_7\in V(CQ_4)$. We claim that $d(v_6)\geq3$. Otherwise, $d(v_6)=2$.  Then $v_6\notin V(Q_4)$,  and thus $G[(V(Q_4)\setminus \{v\})\cup \{v_6\}]$ contains a  copy of $K^3_3$, a contradiction. Thus $e(G)\geq \frac{4+(2+|N_2|)+(2|N_2|+1)+3(n-3-|N_2|)}{2}=\frac{3n-2}{2}$ by Claim 3. Now we show  $e(G)\neq\frac{3n-2}{2}$.
				
				If $e(G)=\frac{3n-2}{2}$, then there exists a unique vertex $v_6\in N_2$ with $d(v_6)=3$ such that   $v_6v_7\in E(G)$ for some  $v_7\in N_3$ and $d(w)=3$ for each $w\in \bigcup\limits_{i=3}^{\varepsilon(v)}N_i$. Clearly, $v_6\in V(Q_4)$, otherwise, $G[(V(Q_4)\setminus \{v\})\cup \{v_6\}]$ contains a  copy of $K^3_3$, a contradiction.
				
				If $v_6\in V(Q_4)\setminus V(CQ_4)$, then there exist $v_8,v_9\in\bigcup\limits_{i=3}^{\varepsilon(v)} N_i$ such that $V(CQ_4)=\{v_7,v_8,v_9\}$ and $v_8v_6\in E(Q_4)$ since $d_{Q_4}(v_8)=d_{Q_4}(v_9)=3$,  which implies that $G+vv_9$ contains no  copy of $K^3_3$, a contradiction.

				If $v_6\in V(CQ_4)$, then we assume $V(CQ_4)=\{v_6,v_7,v_8\}$ such that $v_7,v_8\in N_3$ and $d(v_8)=3$ since there exists $v_9\in (V(Q_4)\setminus V(CQ_4))\cap (\bigcup\limits_{i=2}^{\varepsilon(v)} N_i)$ satisfying $v_8v_9\in E(Q_4)$. If $v_9\in N_2$, then $N(v_9)=\{v_2,v_8\}$, and  $v_9\not\sim N(v_8)\setminus\{v_9\}$, a contradiction to (i) of Lemma \ref{l3}.  If $v_9\in N_3$, then $d(v_9)=3$ and there exists $v_{10}\in N_2\setminus\{v_6\}$ such that $d(v_{10})=2$ and  $v_{10}v_9\in E(G)$ since $N(v_6)=\{v_2,v_7,v_8\}$, which implies that $v_{10}\not\sim N(v_9)\setminus\{v_{10}\}$, a contradiction to (i) of Lemma \ref{l3}. If $v_9\in N_4$, then $d(v_9)=3$ and $v_7v_9\in E(G)$ since $d(v_7)=3$ and $G$ is $K^3_3$-free, which implies that $G+vv_9$ contains no  copy of $K^3_3$, a contradiction. 
				
				Therefore, $e(G)\neq \frac{3n-2}{2}$, and then $e(G)\geq\frac{3n-1}{2}$.
				

				\textbf{Case 3.} $v$ is  a type-III vertex.
				
				Let  $N_0=\{v\}$. Then $N_1=\{v_1,v_2\}$, $v_1v_2\notin E(G)$ with $d(v_1)\geq3,d(v_2)\geq3$ since $v$ is  a type-III vertex, and $V(G)=\bigcup\limits_{i=0}^{\varepsilon(v)}N_i$, which implies that $2\leq\varepsilon(v)\leq 4$ by Lemma \ref{l1}. Clearly,  $v\notin V(CK^2_3)$ for any copy of $K^2_3$ in  $G$.

				\textbf{Claim 4.} Let $V_1=\{u\in N_2\mid uv_1,uv_2\in E(G)\}$. Then $|V_1|\geq1$ and $d(u)\geq3$ for each $u\in V_1$.
				
				\begin{proof}
					Let $Q_5$ be a copy of $K^3_3$ in $G+v_1v_2$. Then $v_1,v_2\in V(CQ_5)$. Otherwise, we may assume $d_{Q_5}(v_1)=1$ and $v_2\in V(CQ_5)$, then $v\notin V(Q_5)$ since $N(v)=\{v_1,v_2\}$, and thus $G[(V(Q_5)\setminus \{v_1\})\cup \{v\}]$ contains  a copy of $K^3_3$, a contradiction. Thus there exists $u\in V(CQ_5)\setminus \{v_1,v_2\}$, which implies that $u\in V_1$ and $d(u)\geq d_{Q_5}(u)\geq3$. 
					
					If $|V_1|\geq2$, then there exists $w\in V_1\setminus\{ u\}$ with $wv_1,wv_2\in E(G)$. Let $Q_6$ be a copy of $K^3_3$ in $G+wv$. Then $d_{Q_6}(w)\neq1$ since $v_1v_2\notin E(G)$ and $N(v)=\{v_1,v_2\}$. If $d_{Q_6}(v)=1$ and $w\in V(CQ_6)$, then $d(w)\geq3$ since $v_1v_2\notin E(G)$. If $w,v\in V(CQ_6)$, then we may assume $V(CQ_6)=\{w,v,v_1\}$, and thus $v_2v\in E(Q_6)$, which implies $d(w)\geq3$.
					
					We complete the proof of the claim.
				\end{proof}

				Now we show $e(G)\geq\frac{3n-1}{2}$ by the following two subcases.
				
				\textbf{Subcase 3.1.} $\varepsilon(v)=2$.
				
				Clearly, we have $V(G)=\bigcup\limits_{i=0}^{2}N_i$.  
				
					If $|V_1|\geq2$, then $e(G)\geq \frac{4+(n-3+2)+(2(n-5)+6)}{2}=\frac{3n-1}{2}$ by Claim 4.

				If $|V_1|=1$ and $d(w)\geq3$ for each $w\in V(G)\setminus\{v\}$, then $e(G)\geq \frac{2+3(n-1)}{2}=\frac{3n-1}{2}$.
				
				If $|V_1|=1$ and there exists vertex $w\in N_2$ with $d(w)=2$, then  $w$ is a type-I, type-II,  type-III, or   type-IV vertex by Lemma \ref{ll1}. Without loss of generality, let $V_1=\{u\}$. Now we show $e(G)\geq \frac{3n-1}{2}$ by the following arguments.

				If  $w$    is a type-I or type-II vertex, then $e(G)\geq \frac{3n-1}{2}$ by Cases 1-2 and $G\not\cong K_1\vee\frac{n-1}{2}K_2$. 
				
				If  $w$    is a type-III vertex, we may assume $N(w)=\{v_1,w_1\}$, where $w_1\in N_2$ and $w_1v_1\notin E(G)$. Then $w_1v_2\in E(G)$ with $d(w_1)\geq3$, and thus $uw_1\notin E(G)$ since $G$ is $K^3_3$-free. Let $Q_7$ be a copy of $K^3_3$ in $G+w_1v_1$. Then $v_1,w_1\in V(CQ_7)$; otherwise, $v_1\notin V(CQ_7)$ or $w_1\notin V(CQ_7)$. If $v_1\notin V(CQ_7)$, then $w\notin V(Q_7)$ since $N(w)=\{v_1,w_1\}$, and thus  $G[(V(Q_7)\setminus \{v_1\})\cup \{w\}]$ contains  a copy of $K^3_3$, a contradiction. Similarly, if $w_1\notin V(CQ_7)$, then $G[(V(Q_7)\setminus \{w_1\})\cup \{w\}]$ contains  a copy of $K^3_3$, a contradiction. Thus there exists $w_2\in N_2\setminus\{w_1,u\}$ such that $ V(CQ_7)=\{v_1,w_1,w_2\}$ with $d(w_2)\geq3$. Therefore, $e(G)\geq \frac{4+(n-2)+(2(n-6)+9)}{2}=\frac{3n-1}{2}$ by the above arguments and Claim 4.

				If  any vertex  of degree $2$ in $N_2$ is a type-IV vertex, 
				  then we have $u\in V(CQ_5)\setminus \{v_1,v_2\}$ by the proof of Claim 4, and there exists  $u_8\in N_2$ such that $uu_8,v_iu_8\in E(G)$ and $v_ju_8\notin E(G)$, where $\{i,j\}=\{1,2\}$. Since  $G$ is $K^3_3$-free, we have $d(u_8)=2$. Since $v$ is a type-III vertex, we have $d(v_j)\geq3$, and there exists $u_9\in N_2\setminus\{u,u_8\}$ such that  $u_9v_j\in E(G)$.
				 
				 If $d(v_i)\geq4$, then there exists $u_{10}\in N_2\setminus\{u,u_8,u_9\}$ with $u_{10}v_i\in E(G)$. When $d(u_9)\geq3$ and $d(u_{10})\geq3$, we have $d(u)+d(u_9)+d(u_{10})\geq 9$. When $d(u_p)=2$ and $d(u_q)\geq3$ for $\{p,q\}=\{9,10\}$,  $u_p$ is a  type-IV vertex. Thus
				  $u_pu\in E(G)$ satisfies $d(u)+d(u_q)+d(u_{p})\geq 9$, or there exists $u_{11}\in N_2\setminus\{u,u_8,u_9,u_{10}\}$ with $u_pu_{11}\in E(G)$ and $d(u_{11})\geq3$ such that $d(u)+d(u_q)+d(u_{11})\geq 9$.  When $d(u_9)=d(u_{10})=2$, both $u_9$ and $u_{10}$ are type-IV vertices. Similar to the above arguments, there exist three vertices in $N_2$ such that the sum of degrees of these three vertices is at least $9$. Combining the above arguments, we have  $e(G)\geq \frac{2+n+(2(n-6)+9)}{2}=\frac{3n-1}{2}$.
				  
				  If $d(v_i)=3$, then $d(u)\geq4$. Otherwise, we have $N(u)=\{v_1,v_2,u_8\}$, which implies that $G+v_ju_8\notin E(G)$ contains no a copy of $K^3_3$, a contradiction. Thus $e(G)\geq\frac{2+n+(2(n-4)+4)}{2}=\frac{3n-2}{2}$. We claim that $e(G)\geq \frac{3n-1}{2}$. Otherwise,  we have $d(u)=4$ and $d(u')=2$ for each $u'\in N_2\setminus\{u\}$, then $n=6$ since each $u'$ is a type-IV vertex with $uu'\in E(G)$,
				  which contradicts $n\geq7$.
				
				\textbf{Subcase 3.2.} $3\leq\varepsilon(v)\leq4$.

				
				
				 Let $V_2$ denote the vertex set containing all type-IV vertices in $\bigcup\limits_{i=3}^{\varepsilon(v)}N_i$. Then $d(x)\geq3$ for each $x\in (\bigcup\limits_{i=3}^{\varepsilon(v)}N_i)\setminus V_2$ by (ii) of Lemma \ref{l3} and $\delta(G)=2$.

              \textbf{Subcase 3.2.1.} $|V_2|\geq1$.
				
				Let $x_1\in V_2$. Then $d(x_1)=2$ since $x_1$ is a type-IV vertex, and  $N(x_1)\subseteq N_2$ by (ii) of Lemma \ref{l3}. Suppose that $N(x_1)=\{x_2,x_3\}$, then $x_2x_3\in E(G)$ and $d(x_2)\geq3,d(x_3)\geq3$ since $x_1\in V_2$. By Claim 4, there exists $u\in V_1$ such that $d(u)\geq3$. We claim that $u\notin \{x_2,x_3\}$. Otherwise, $G[\{v_i,x_2,x_3,v,x_1,v_j\}]$ contains  a copy of $K^3_3$ with $\{i,j\}=\{1,2\}$, a contradiction. Thus we have $e(G)=e(G-V_2)+2|V_2|\geq \frac{2+(|N_2|+3)+(2|N_2|+1)+3(n-3-|N_2|-|V_2|)}{2}+2|V_2|=\frac{3n+|V_2|-3}{2}\geq\frac{3n-2}{2}$. If $e(G)=\frac{3n-2}{2}$, then $V_2=\{x_1\}$, $V_1=\{u\}$ with $d(x_2)=d(x_3)=d(u)=3$, and $d(x^*)=2$ for each $x^*\in N_2\setminus\{x_2,x_3,u\}$, which implies that $G+x_1v_1$ contains no a copy of $K^3_3$, a contradiction. Thus we have $e(G)\geq\frac{3n-1}{2}$.
				
				\textbf{Subcase 3.2.2.} $|V_2|=0$.

					\textbf{Claim 5.} There exists $x_4\in \bigcup\limits_{i=2}^{\varepsilon(v)}N_i$ such that $d(x_4)\geq4$.
					
					\begin{proof}
						Suppose to the contrary that  $d(x')\leq3$ for each $x'\in \bigcup\limits_{i=2}^{\varepsilon(v)}N_i$.  Then $d(x)=3$ for each $x\in \bigcup\limits_{i=3}^{\varepsilon(v)}N_i$. Let  $x_5\in N_{\varepsilon(v)}$ and $Q_8$ be a copy of $K^3_3$ in $G+x_5v$. Then $d(x_5)=3$, $d_{Q_8}(v)=1$ and $x_5\in V(CQ_8)$. Without loss of generality, let $N(x_5)=\{x_6,x_7,x_8\}$, and $V(CQ_8)=\{x_5,x_6,x_7\}$. Then  $x_6,x_7\in N_{\varepsilon(v)}\cup N_{\varepsilon(v)-1}$ since $x_5\in N_{\varepsilon(v)}$, and $x_6x_8\in E(G)$ or $x_7x_8\in E(G)$ by (i) of Lemma \ref{l3}. We may assume $x_6x_8\in E(G)$.
						
						If $x_6\in N_{\varepsilon(v)-1}$, then $d(x_6)\geq4$, a contradiction. 
						
						If $x_6,x_7\in N_{\varepsilon(v)}$, then $d(x_6)=d(x_7)=3$ and $x_8\in N_{\varepsilon(v)-1}$ since $N(x_5)=\{x_6,x_7,x_8\}$, and thus $x_7x_8\notin E(G)$ since $d(x_8)\leq3$ and there exists $x_9\in N_{\varepsilon(v)-2}$ such that $x_8x_9\in E(G)$, which implies that $G+vx_7$ contains no  copy of $K^3_3$, a contradiction.

						If $x_7\in N_{\varepsilon(v)-1}$ and $x_6\in N_{\varepsilon(v)}$, then $x_8\in N_{\varepsilon(v)-1}$ and $\varepsilon(v)=3$. Otherwise, $d(v,x_8)\geq3$, which implies that $G+vx_8$ contains no  copy of $K^3_3$, a contradiction. Now we may assume that $x_8v_2,x_7v_i\in E(G)$, then $x_8v_1,x_7v_j\notin E(G)$ since $x_7,x_8\in N_2\setminus V_1$, where $\{i,j\}=\{1,2\}$, and thus $d(x_7)=d(x_8)=3$, which implies that $x_7x_8\notin E(G)$ and $G+v_jx_7$ contains no  copy of $K^3_3$, a contradiction.
						
						Combining the above arguments, Claim 5 holds.
					\end{proof}
					
					By Claim 4 and Claim 5, we have $e(G)\geq \frac{2+(|N_2|+3)+(2|N_2|+1)+3(n-|N_2|-3)+1}{2}\geq \frac{3n-2}{2}$. Clearly, if $e(G)=\frac{3n-2}{2}$, then $V_1=\{u\}$, $3\leq d(u)\leq 4$, $d(u_1)=2$ for any $u_1\in N_2\setminus \{u\}$, and $d(w)=3$ for any $w\in \bigcup\limits_{i=3}^{\varepsilon(v)}N_i$ if $d(u)=4$. Let $u_1\in N_2$ and  $N(u_1)=\{v_i,u_2\}$ with $i=1$ or $2$. Then $u_2\in N_2$, otherwise, $u_2\in N_3$, and thus $u_1\not\sim N(u_2)\setminus\{u_1\}$,  a contradiction to (i) of Lemma \ref{l3}. Thus we have $uw\in E(G)$ for any $w\in N_3$, which implies $|N_3|\in\{1,2\}$ since $d(u)\in \{3,4\}$. 	If $|N_3|=1$, i.e.,  $N_3=\{w\}$,  then $u\not\sim N(w)\setminus\{u\}$, a contradiction to (i) of Lemma \ref{l3}.  If $|N_3|=2$, then $d(u)=4$. We assume that $N_3=\{w,w'\}$, then $d(w)=d(w')=3$. Furthermore, we have $ww'\in E(G)$, otherwise, $u\not\sim N(w)\setminus\{u\}$, a contradiction to (i) of Lemma \ref{l3}. Since $d(w)=d(w')=3$ and $G$ is $K^3_3$-free, we have $|N_4|=1$, which implies that $d(w'')=2$ for $w''\in N_4$,  contradicting $d(w'')=3$.

					\textbf{Case 4.}  Every vertex  of degree $2$ in $G$ is a type-IV vertex.

					Let $Z=\{z\in V(G)\mid d(z)=2\}$, $W=\{w\in N(z)\mid z\in Z\}$, $W_i$	be a nonempty subset of $W$	such that  there exists a vertex $z_1$ with $N(z_1)=\{w_1, w_2\}$ for any $w_1 w_2\in E(G[W_i])$, and 	$G[W_i]$	be a maximal connected component of $G[W]$, 
					 where $i\in \{1,\ldots,r\}$. Then  $W_1\cup\cdots\cup W_r$ is a partition of $W$. Furthermore, let $Z_i=\{z\in Z\mid N(z)\subseteq W_i\}$ for $i\in \{1,\ldots,r\}$. Then $Z_1\cup\cdots\cup Z_r$ is a partition of $Z$, and $|Z_i|\geq |E(G[W_i])|\geq |W_i|-1$ for each $i\in \{1,\ldots,r\}$. 
					 
					 Since $z_i$ is a type-IV vertex, we have  $d(w_{i1})\geq3$ and $d(w_{i2})\geq3$ for any $z_i\in Z$ with $N(z_i)=\{w_{i1},w_{i2}\}$. If $w_{i2}=w_{j1}$ and $\{w_{i1},w_{i2}\}\neq\{w_{j1},w_{j2}\}$  for $i\neq j$, we have $d(w_{i2})\geq4$, $d(w_{i2})+d(z_i)\geq3\times2$ and $d(w_{i1})+d(w_{i2})+d(w_{j2})+d(z_i)+d(z_{j})\geq3\times5-1$. Moreover, if $\{w_{i1},w_{i2}\}=\{w_{j1},w_{j2}\}$  for $i\neq j$, then $d(w_{i1})+d(w_{i2})+d(z_i)+d(z_{j})\geq3\times4$. Thus,  we have 
					\begin{equation}
						\sum\limits_{z\in Z_i}d(z)+\sum\limits_{w\in W_i}d(w)\geq\begin{cases}
							3(|Z_i|+|W_i|),&\text{if}\ |Z_i|\geq |W_i| \ \text{or} \ G[W_i]\not\cong P_{|W_i|}; \\
							3(|Z_i|+|W_i|)-1,&\text{if} \ |Z_i|= |W_i|-1 \ \text{and} \ G[W_i]\cong P_{|W_i|}.\label{e}
						\end{cases}
					\end{equation}

					\textbf{Claim 6.} There exists at most one $i\in\{1,\ldots,r\}$ such that $\sum\limits_{z\in Z_i}d(z)+\sum\limits_{w\in W_i}d(w)=3(|Z_i|+|W_i|)-1$.
				
				\begin{proof}
					Suppose to the contrary that $\sum\limits_{z\in Z_1}d(z)+\sum\limits_{w\in W_1}d(w)=3(|Z_1|+|W_1|)-1$ and  $\sum\limits_{z\in Z_2}d(z)+\sum\limits_{w\in W_2}d(w)=3(|Z_2|+|W_2|)-1$. Then for $j\in\{1,2\}$, we have $|Z_j|= |W_j|-1$ and $G[W_j]\cong P_{|W_j|}$ by \eqref{e} and $|Z_i|\geq |W_i|-1$ for each $i\in \{1,\ldots,r\}$. Without loss of generality, let $W_1=\{w_1,w_2,\ldots,w_b\}$ with $w_kw_{k+1}\in E(G)$,  $Z_1=\{z_1,\ldots,z_{b-1}\}$ with $N(z_k)=\{w_k,w_{k+1}\}$, where $b\geq2$ and $k\in\{1,\ldots,b-1\}$, $W_2=\{\mathsf{w}_1,\mathsf{w}_2,\ldots,\mathsf{w}_c\}$ with $\mathsf{w}_l\mathsf{w}_{l+1}\in E(G)$,  $Z_2=\{\mathsf{z}_1,\ldots,\mathsf{z}_{c-1}\}$ with $N(\mathsf{z}_l)=\{\mathsf{w}_l,\mathsf{w}_{l+1}\}$, where $c\geq2$ and $l\in\{1,\ldots,c-1\}$. Then $d(w_1)=d(w_b)=d(\mathsf{w}_1)=d(\mathsf{w}_c)=3$ and $d(w_p)=d(\mathsf{w}_q)=4$ for all $p\in\{1,\ldots,b\}\setminus\{1,b\}$ and  $q\in\{1,\ldots,c\}\setminus\{1,c\}$.
					
					If $b\geq3$ and $c\geq3$, then $G+w_2\mathsf{w}_2$  contains no  copy of $K^3_3$, a contradiction.

					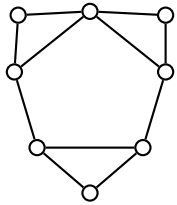
\begin{figure}[ht]
						\centering
						\begin{tikzpicture}[
							scale=1,
							round/.style={circle, draw, thick, minimum size=6mm},  smallv/.style={
								circle, 
								draw, 
								thick, 
								inner sep=0pt,     
								minimum size=2mm 
							}
							]
							
							\node[smallv] (v1) at (0, 0)  {};
							\node[smallv] (v2) at (1.4, 0)  {};
							\node[smallv] (v3) at (0.7, -0.6)  {};
							\node[smallv] (v4) at (1.7, 1)  {};
							\node[smallv] (v5) at (0.7, 1.8)  {};
							\node[smallv] (v6) at (-0.3, 1)  {};
							\node[smallv] (v7) at (1.7, 1.75)  {};
							\node[smallv] (v8) at (-0.25, 1.75)  {};
							\draw[thick] (v1)--(v2);
							\draw[thick] (v3)--(v2);
							\draw[thick] (v1)--(v3);
							\draw[thick] (v2)--(v4);
							\draw[thick] (v4)--(v5);
							\draw[thick] (v5)--(v6);
							\draw[thick] (v6)--(v1);
							\draw[thick] (v4)--(v7);
							\draw[thick] (v5)--(v7);
							\draw[thick] (v5)--(v8);
							\draw[thick] (v6)--(v8);
						\end{tikzpicture}
						\caption{ Graph  $G_0$.}\label{f2}
					\end{figure}
					
					If $b\geq3$ and $c=2$, then $W_2=\{\mathsf{w}_1,\mathsf{w}_2\}$ with $d(\mathsf{w}_1)=d(\mathsf{w}_2)=3$. Let $Q_9$ be a copy of $K^3_3$ in $G+w_2\mathsf{w}_1$. Then $w_2,\mathsf{w}_1\in V(CQ_9)$, and thus $w_1\mathsf{w}_1\in E(G)$ or $w_3\mathsf{w}_1\in E(G)$ (if $b=3$). Similarly, we have  $w_1\mathsf{w}_2\in E(G)$ or $w_3\mathsf{w}_2\in E(G)$ (if $b=3$) since $G+w_2\mathsf{w}_2$ contains a copy of $K^3_3$. Furthermore, we have $w_1\mathsf{w}_1,w_3\mathsf{w}_2\in E(G)$ and thus  $b=3$ by $d(w_1)=3$, or $w_1\mathsf{w}_2,w_3\mathsf{w}_1\in E(G)$ and thus  $b=3$ since $d(w_1)=3$.
					 We may assume $w_1\mathsf{w}_1,w_3\mathsf{w}_2\in E(G)$. Clearly, $G\cong G_0$, where $G_0$ is shown in Figure \ref{f2}, but $G_0$ is not a $K^3_3$-saturated graph,  a contradiction.
					
					If $b=c=2$, then $W_1=\{w_1,w_2\}$, $W_2=\{\mathsf{w}_1,\mathsf{w}_2\}$, $Z_1=\{z_1\}$, $Z_2=\{\mathsf{z}_1\}$ and $d(w_1)=d(w_2)=d(\mathsf{w}_1)=d(\mathsf{w}_2)=3$. If $G[\{w_1,w_2,\mathsf{w}_1,\mathsf{w}_2\}]\cong2K_2$, then there exists a vertex $w^*\in V(G)\setminus (Z_1\cup Z_2\cup W_1\cup W_2)$ such that $N(w_1)=\{z_1,w_2,w^*\}$. Let $Q_{10}$ be a copy of $K^3_3$ in $G+w_1\mathsf{w}_1$. Then $w_1,\mathsf{w}_1\in V(CQ_{10})$, and thus $V(CQ_{10})=\{w_1,\mathsf{w}_1, w^*\}$ since $G[\{w_1,w_2,\mathsf{w}_1,\mathsf{w}_2\}]=2K_2$, which implies $N(\mathsf{w}_1)=\{\mathsf{z}_1,\mathsf{w}_2,w^*\}$. Similarly, let $Q_{11}$ be a copy of $K^3_3$ in $G+w_2\mathsf{w}_1$. Then $ V(CQ_{11})=\{w_2,\mathsf{w}_1,w^*\}$, and thus $N(w_2)=\{z_1,w_1,w^*\}$ since $d(w_2)=d(\mathsf{w}_1)=3$, which implies that $G+z_1w^*$ contains no  copy of $K^3_3$, a contradiction. If $G[\{w_1,w_2,\mathsf{w}_1,\mathsf{w}_2\}]\cong P_4$, we assume that $P_4=w_1w_2\mathsf{w}_1\mathsf{w}_2$. Thus $N(w_2)=\{z_1,w_1,\mathsf{w}_1\}$ and $N(\mathsf{w}_1)=\{\mathsf{z}_1,w_2,\mathsf{w}_2\}$, which implies  that $G+z_1\mathsf{w}_1$ contains no  copy of $K^3_3$, a contradiction. If $G[\{w_1,w_2,\mathsf{w}_1,\mathsf{w}_2\}]\cong C_4$, then $n=6$ since $d(w_1)=d(w_2)=d(\mathsf{w}_1)=d(\mathsf{w}_2)=3$,  which  contradicts $n\geq7$.
					
					Combining the above arguments, Claim 6 holds.
				\end{proof}
				 By Claim 6, we have 
				\begin{align}
					e(G)&=\frac{\sum\limits_{i=1}^{r}(\sum\limits_{z\in Z_i}d(z)+\sum\limits_{w\in W_i}d(w))+\sum\limits_{w''\in V(G)\setminus (Z\cup W)}d(w'')}{2}\nonumber \\ &\geq \frac{(3(|Z|+|W|)-1)+3(n-|Z|-|W|)}{2}=\frac{3n-1}{2}.\nonumber
				\end{align}
				
				Combining Cases 1-4, we have  $G\cong K_1\vee \frac{n-1}{2}K_2$ with odd $n$ or $e(G)\geq \frac{3n-1}{2}$. Furthermore, both $K_1\vee \frac{n-1}{2}K_2$  (odd $n$) and 
				$K_1\vee (\frac{n-4}{2}K_2\cup K_3)$  (even $n$) are $K^3_3$-saturated graphs of order $n$, and $e(G)=\begin{cases}
					\frac{3n-3}{2},&\text{if}\ G\cong K_1\vee \frac{n-1}{2}K_2  \\
					\frac{3n}{2},&\text{if}\ G\cong K_1\vee (\frac{n-4}{2}K_2\cup K_3)
				\end{cases}$. Therefore,  the conclusion holds.
		\end{proof}

		\begin{lemma}\label{ll2}
			Let $G$ be a connected $K^3_3$-saturated graph of order $n\geq6$. Then there exists $u\in V(G)$ such that  $d_{K^2_3} (u)\neq2$ for any  subgraph $K^2_3$ of $G$.
		\end{lemma}
		
		\begin{proof}
			Suppose otherwise, then there exists a subgraph $K^2_3$ of $G$	such that $d_{K^2_3}(v)=2$ for any vertex $v\in V(G)$. 
			
			Let $v_1\in V(G)$. Then there exists a copy $R_1$ of $K^2_3$ in $G$ such that $d_{R_1}(v_1)=2$. We may assume $V(CR_1)=\{v_1,v_2,v_3\}$ and $v_2u_1,v_3u_2\in E(R_1)$, where $u_1,u_2\in V(G)\setminus\{v_1,v_2,v_3\}$. Furthermore, there exist two copies $R_2,R_3$ of $K^2_3$ in $G$ such that $d_{R_2}(v_2)=2$ and $d_{R_2}(v_3)=2$ since $v_2,v_3\in V(G)$.
			
			Clearly, we have $v_1u\notin E(G)$ for any $u\in V(G)\setminus V(R_1)$, otherwise, $G[V(R_1)\cup \{u\}]$ contains  a copy of $K^3_3$, a contradiction. Then $N(v_1)\subseteq\{v_2,v_3,u_1,u_2\}$. Now we complete the proof by the following three cases.
			
			\textbf{Case 1.} $V(CR_2)\cap V(CR_1)=\{v_2\}$.
			
			In this case, we have $v_1,v_3\in V(R_2)$. Otherwise,  $v_i\notin V(R_2)$  for $i=1$ or $3$, then $G[V(R_2)\cup \{v_i\}]$ contains a copy of $K^3_3$, a contradiction. Suppose that $V(R_2)=\{v_2,w_1,w_2\}$, then $w_1v_1, w_2v_3\in E(G)$ or $w_1v_3, w_2v_1\in E(G)$ since  $v_1,v_3\in V(R_2)$, which implies that $P_4$ is a subgraph of $G[N(v_2)]$, a contradiction to Observation \ref{o}. 
			
				\textbf{Case 2.} $V(CR_2)\cap V(CR_1)=\{v_2,v_3\}$.
			
		Then there exists $v_4\in V(G)\setminus\{v_1,v_2,v_3\}$ such that $V(CR_2)=\{v_2,v_3,v_4\}$. Clearly, $v_1\in V(R_2)$, otherwise, $G[V(R_2)\cup\{v_1\}]$ contains  a copy of $K^3_3$, a contradiction.  Moreover,  we have $v_1v_2\notin E(R_2)$ since $d_{R_2}(v_2)=2$. Then $v_1v_4\in E(R_2)$ or $v_1v_3\in E(R_2)$. 
			
			 If $v_1v_4\in E(R_2)$, then there exists $v_5\in V(G)\setminus \{v_1,v_2,v_3,v_4\}$ such that $v_3v_5\in E(R_2)$. For any $w\in V(G)\setminus\{v_1,v_2,v_3,v_4,v_5\}$, we have $v_1w,v_2w,v_4w\notin E(G)$, otherwise, $G$ contains a copy of $K^3_3$, a contradiction. Since $n\geq6$ and $G$ is connected, there exists $w'\in V(G)\setminus\{v_1,v_2,v_3,v_4,v_5\}$ such that $w'v_3\in E(G)$ or $w'v_5\in E(G)$. Whether $w'v_3\in E(G)$ or $w'v_5\in E(G)$, we have $v_5v_1,v_5v_2,v_5v_4\notin E(G)$ since $G$ is $K^3_3$-free. Therefore, $d(v_1)=d(v_2)=d(v_4)=3$.  It follows that $V(CR_3)\cap \{v_1,v_2,v_3,v_4\}=\{v_3\}$, which implies that  $G[V(R_3)\cup\{v_1\}]$ contains  a copy of $K^3_3$, a contradiction. 
			 
			 If $v_1v_3\in E(R_2)$, then there exists $v_6\in V(G)\setminus \{v_1,v_2,v_3,v_4\}$ such that $v_4v_6\in E(R_2)$. Since $G$ is $K^3_3$-free, we have $v_iv_7\notin E(G)$ for any $v_7\in V(G)\setminus \{v_1,v_2,v_3,v_4,v_6\}$ and $i\in \{2,3\}$, and thus $v_jv_6\in E(R_1)$ for $j=2$ or $3$  since $R_1$ is a copy of $K^2_3$, which implies that $P_4$ is a subgraph of $G[N(v_j)]$, a contradiction to Observation \ref{o}. 
			 
			 	\textbf{Case 3.} $V(CR_2)\cap V(CR_1)=\{v_2,v_1\}$.
			 	
			 	Similar to Case 2, there exists $v'_4\in V(G)\setminus\{v_1,v_2,v_3\}$ such that $V(CR_2)=\{v_2,v_3,v'_4\}$, and $v_3v'_4\in E(R_2)$ or $v_3v_1\in E(R_2)$.
			 	
			 	If $v_3v'_4\in E(R_2)$, then $d(v_3)=d(v_2)=d(v'_4)=3$ by a similar argument, which contradicts that $R_1$ is a copy  of $K^2_3$ in $G$ such that $d_{R_1}(v_1)=2$ and $V(CR_1)=\{v_1,v_2,v_3\}$.
			 	
			 	If $v_3v_1\in E(R_2)$, then there exists $v'_6\in V(G)\setminus \{v_1,v_2,v_3,v'_4\}$ such that $v'_4v'_6\in E(R_2)$. Since    $G$ is $K^3_3$-free, we have $v_jv'_7\notin E(G)$ for any $v'_7\in V(G)\setminus \{v_1,v_2,v_3,v'_4,v'_6\}$ and $j\in \{1,2\}$. Furthermore, we have $v'_6v_j\notin E(G)$, otherwise, $P_4$ is a subgraph of $G[N(v_j)]$, a contradiction to Observation \ref{o}. Thus   $d(v_1)=d(v_2)=3$, which implies $|V(CR_3)\cap V(CR_1)|=2$. If $V(CR_3)\cap V(CR_1)=\{v_3,v_1\}$, then $V(CR_3)=\{v_3,v_1,v'_4\}$ since $d(v_1)=3$. Clearly, $v_2v'_4\in E(R_1)$ since $d(v_2)=3$. Then there exists $v'\in V(G)\setminus\{v_1,v_2,v_3,v'_4\}$ such that $v'v_3\in E(R_1)$.  Thus $P_4$ is a subgraph of $G[N(v_3)]$ if $v'=v'_6$, a contradiction to Observation \ref{o}; and $G$ contains a copy of $K^3_3$ if $v'\neq v'_6$, also a contradiction. If $V(CR_3)\cap V(CR_1)=\{v_3,v_2\}$, we can obtain a contradiction by similar proof, and we omit the details.

			\textbf{Case 4.} $V(CR_2)\cap V(CR_1)=\{v_1,v_2,v_3\}$.
			
			 By a similar argument to Cases 1-3, we obtain that  $V(CR_3)\cap V(CR_1)=\{v_1,v_2,v_3\}$. By $N(v_1)\subseteq\{v_2,v_3,u_1,u_2\}$, we have $v_1u_1\in E(G)$ or $v_1u_2\in E(R_1)$.
			 If $v_1u_2\in E(R_1)$, then $v_3u_1\in E(R_1)$, which implies that $P_4$ is a subgraph of $G[N(v_3)]$,  a contradiction to  Observation \ref{o}. Therefore,  $v_1u_1\in E(R_1)$. Clearly, $u_2\in V(R_3)$ since $v_3u_2\in E(G)$, and then $u_2v_i\in E(R_3)$ for $i=1$ or $2$, which implies that $P_4$ is a subgraph of $G[N(v_i)]$,  a contradiction to  Observation \ref{o}.

			Combining the above arguments, we complete the proof.
		\end{proof}
		
		\hspace*{6mm}
		
		\noindent\textbf{\textit{Proof of Theorem \ref{t1}.}}  Let $G$ be a $K^3_3$-saturated graph of order $n$. Then we consider the following two cases.
		
			\textbf{Case 1.} $G$ is connected.
			
			If $\delta(G)=1$, then $e(G)>\frac{3n}{2}$ by Theorem \ref{l2}.
			
			If  $\delta(G)=2$, then $e(G)\geq\begin{cases}
				\frac{3n-3}{2},&\text{if}\ n \ \text{is odd} \\
				\frac{3n}{2},&\text{if}\ n\ \text{is even}
			\end{cases}$ and $e(G)=\frac{3n-3}{2}$ if and only if $G\cong K_1\vee \frac{n-1}{2}K_2$ by Theorem \ref{p1}.
			
			If $\delta(G)\geq3$, then $e(G)\geq \frac{3n}{2}$.
			
				\textbf{Case 2.} $G$ is disconnected.
				
				Then $G\cong \bigcup\limits_{i=1}^{k}G_i$, where $k\geq2$ and each $G_i$ is  connected. Clearly, every $G_i$ is either a $K^3_3$-saturated graph with $|V(G_i)|\geq6$, or isomorphic to $K_r$ with $1\leq r\leq5$.

				Firstly, if $G_1$ is a $K^3_3$-saturated graph with $|V(G_1)|\geq6$, then $G_i\cong K_5$ for  any $i\in\{2,\ldots,k\}$. Otherwise, $G_2$ is either a $K^3_3$-saturated graph with $|V(G_2)|\geq6$ or isomorphic to 
				$K_r$ with $1\leq r\leq 4$. By Lemma \ref{ll2} and $|V(K_r)|\leq4$, there exist $v_1\in V(G_1)$ and $v_2\in V(G_2)$ such that $d_{K^2_3}(v_1)\neq2$ for any  subgraph $K^2_3$ of $G_1$ and $d_{K^2_3}(v_2)\neq2$ for any  subgraph $K^2_3$ of $G_2$, which implies that $G+v_1v_2$ contains no copy of $K^3_3$, a contradiction.
				 Therefore, we have $G\cong G_1\cup (k-1)K_5$, and $e(G)\geq \frac{3(n-5k+5)-3}{2}+10(k-1)=\frac{3n+5k-8}{2}\geq\frac{3n+2}{2}$ by Case 1 and $k\geq2$.

			Secondly, if $G_i$ is not a $K^3_3$-saturated graph with $|V(G_i)|\geq6$ for any $i\in \{1,\ldots,k\}$, then $G\cong (k-p)K_5 \cup p K_r$ with $1\leq r\leq 4$, $p\in \{0,1\}$. Thus we have $e(G)= \frac{4(n-pr)+pr(r-1)}{2}=\frac{4n-pr(5-r)}{2}\geq \frac{3n+1}{2}$ since $n\geq7$.
			
			Combining the above arguments, we complete the proof.  $\hfill\blacksquare$

				\section{Conclusion}
				
			\hspace{1.5em}	In this paper, we study $K^2_t$-saturated graphs with $t\geq4$ and $K_3^3$-saturated graphs. About $K_3^3$-saturated graphs,  we only characterize $\mathrm{SAT}(n,K^3_3)$ when $n$ is odd. For even $n$, we only give one extremal graph. Although a complete characterization of $\mathrm{SAT}(n,K^3_3)$ can be achieved by the method in Theorem \ref{p1}, it would make the proof rather lengthy and cumbersome. Meanwhile, we conjecture that the following result holds.
				\begin{conjecture}
			There exists a constant $k^*$ such that  \\ $\mathrm{SAT}(n,K^3_3)=\begin{cases}
			\{K_1\vee \frac{n-1}{2}K_2\},&\text{if}\ n \ \text{is odd}; \\
				\{K_1\vee (\frac{n-4}{2}K_2\cup K_3)\},&\text{if}\ n\ \text{is even}
			\end{cases}$ whenever $n\geq k^*$.
			\end{conjecture}

			For $t=4$, we completely characterize $\mathrm{SAT}(n,K^2_4)$, and for $t\geq5$, we characterize $\mathrm{CSAT}(n,K^2_t)$. However, $\mathrm{SAT}(n,K^2_t)$  is difficult to determine for $t\geq5$, which depends on the relationship among $n$, $t$ and $f$ given in Theorem \ref{t4}. Furthermore, by the proof of Theorem \ref{t4}, we easily obtain the following conclusion.
			\begin{remark}
Let $G_1,f',r',f,n$ be as defined in Theorem \ref{t4} and its proof. Then $\mathrm{SAT}(n,K^2_t) \\ \subseteq \{K_{f'}\cup \frac{n-f'}{t+1}K_{t+1}, \mathcal{SM}^t_{|V(G_1)|}(K_{r'})\cup \frac{n-|V(G_1)|}{t+1}K_{t+1},  \mathcal{SM}^t_{n}(K_{f})\}$ with $\{f',r',f\}\subseteq\{\frac{t-1}{2},\frac{t}{2},\\ \frac{t+1}{2},\frac{t+2}{2},\frac{t+3}{2}\}$.
			\end{remark}
			
			Based on the results of this paper, we propose the following conjecture after careful consideration.
			\begin{conjecture}
				Let $2\leq s\leq t-2$. Then there exists a function $f(s,t)$ independent of $n$ such that $\mathrm{sat}(n,K^s_t)=\frac{t+s-2}{2}n+f(s,t)$.
			\end{conjecture}
			
				\section*{\bf Funding}
			
			\hspace{1.5em} This work is  supported by the National Natural Science Foundation of China (Grant Nos. 12371347, 12271337).
			
			\section*{Declarations}
			
			\textbf{Conflict of interest}\  The authors declare that they have no known competing financial interests or personal relationships that could have appeared to influence the work reported in this paper.
			
			\vskip 0.5em
			
			\noindent\textbf{Data availability} \  No data was used for the research described in the article.

\end{document}